\documentclass[12pt,a4paper]{amsart}

\usepackage{amsmath, amsthm, amssymb}
\usepackage{graphicx}
\usepackage[margin = 2.5cm]{geometry}
\usepackage{mathrsfs}

\theoremstyle{definition}

\newtheorem{remark}{Remark}
\theoremstyle{theorem}
\newtheorem{lemma}{\bf Lemma}
\newtheorem{proposition}{\bf Proposition}
\newtheorem{theorem}{\bf Theorem}

\usepackage{algorithm}
\usepackage{algorithmic}

\makeatletter
\renewcommand{\p@algorithm}{\arabic{algorithm}\expandafter\@gobble}
\makeatother 
\newcommand{\PARAMETERS}{\item[\textbf{Parameters:}]}

\newcounter{step}[algorithm]
\newcommand\STEP[2][\(\triangleright\)]{%
	\refstepcounter{step}
	\vskip 0.25\baselineskip
	\item[]\hskip -\algorithmicindent #1 \textbf{Step \arabic{step}}%
	\ifthenelse{\equal{\unexpanded{#2}}{}}{}{ (\texttt{#2})}%
	\textbf{.}%
}
\newenvironment{algo}{\algo}{}
\def\algo#1\end{%
	\noindent\fbox{%
	\begin{minipage}[b]{\dimexpr\columnwidth-\algorithmicindent\relax}
	\begin{algorithmic}
	#1
	\end{algorithmic}
	\end{minipage}
	}%
\end}
\newcommand{\x}{\mathbf x}

\graphicspath{{./}{./src/}{./build/}}

\usepackage[pdf]{pstricks} 
\usepackage{pstricks-add}
\usepackage{pst-pdf}
\usepackage{pstricks}
{\rm}

\newcommand{\afunc}[1]{\operatorname{\mathsf{#1}}}

\begin{document}

\title[Boundary control of partial differential equations]{Boundary control of partial differential equations using frequency domain optimization techniques}
\author{Pierre Apkarian and Dominikus Noll}  
\thanks{Pierre Apkarian is with ONERA,  2 Av. Ed. Belin, 31055, Toulouse, France 
       {\tt\small Pierre.Apkarian@onera.fr}}        
\thanks{Dominikus Noll is with Universit\'e Paul Sabatier, Institut de Math\'ematiques, 118,
route de Narbonne, 31062 Toulouse, France         
{\tt\small dominikus.noll@math.univ-toulouse.fr}}%
\date{}

\maketitle

\begin{abstract} 
We present a frequency domain based $H_\infty$-control strategy to solve boundary control problems for systems governed by
parabolic or hyperbolic partial differential equation, where controllers are constrained to be physically implementable and of simple
structure suited for practical applications. The efficiency of our technique is demonstrated by controlling a reaction-diffusion equation with input delay,
and a wave equation with boundary anti-damping. 

\vspace{.2cm}\noindent
\end{abstract}

\section{Introduction}
A recurrent issue in system control is whether, or to what extent, frequency-domain based 
$H_\infty$-control strategies originally developed for real-rational systems expand to infinite-dimensional processes. Success in
rendering  $H_\infty$-optimization fit to provide
practically implementable controllers for infinite-dimensional systems
should substantially foster the acceptance of PDE-modeling as a tool  for control.

In response to this quest, we 
present a frequency domain based method to control infinite-dimensional LTI-systems, which is in particular suited for 
$H_\infty$-boundary control of parabolic and hyperbolic partial differential equations. Our method leads to practically implementable structured output feedback controllers
for PDEs in such a way that the typical work-flow in control design is respected.

After briefly outlining our method,
we will  apply it in more detail to two infinite-dimensional $H_\infty$-control
problems: boundary control of a reaction-diffusion system with input delay, as discussed in \cite{prieur_trelat:17},
and  boundary control of an anti-stable wave equation
to control noise and disturbance effects on duct combustion dynamics in a drilling pipe system \cite[p.6]{Fridman:2015}, \cite{bresch_krstic:14}.
While the first study leads to a parabolic equation of retarded type, the second study leads to a system of neutral type, which
poses new challenges to our frequency approach.

The structure of the paper is as follows. In section \ref{outline}
we give the principal steps of our method. Stability is discussed in section \ref{sect_stability}, the role of the Nyquist test in optimization
in section  \ref{G0}, its implementation in section \ref{nyquist}.          Sampling for performance is addressed in section \ref{barrier}.
In section \ref{parabolic}, we discuss the application of our method to a reaction-diffusion equation, and in section \ref{sect_wave}
to a wave equation.

\section{Outline of the method}
\label{outline}
We start out with an infinite-dimensional LTI-system represented by a
transfer function $G(s)$ with $p$ inputs and $m$ outputs, assumed well-posed in the sense of Salamon-Weiss \cite{Salamon:1987,CurtainZwart:1995,ChengMorris:2003,CurtainWeiss:1989}.
As principal application, we consider the case of
a linearized parabolic or hyperbolic boundary control problem in state-space form
\begin{align}
\label{boundary}
\Gamma : \left\{
\begin{matrix}
\dot{x} &\!\!\!= \!\!\!&Ax \\
Px &\!\!\!=\!\!\!& u\\
y &\!\!\!=\!\!\!& Cx
\end{matrix}
\right.
\end{align}
with operators $A\in L(X,H), P\in L(X,\mathbb R^p), C\in L(X,\mathbb R^m)$ on Hilbert spaces $H,X$ and finite-dimensional input and output spaces, where $X$ is densely embedded in $H$.
% and $D(A) \subset D(P)$. {\color{red} ?????}
%We assume that as in  Salamon \cite{Salamon:1987,CurtainZwart:1995,ChengMorris:2003,CurtainWeiss:1989}, the boundary control problem can 
%in principle be transformed into
%a well-posed system $\Sigma_\gamma$ in the sense of Salamon and Weiss. 
Then under natural assumptions specified in \cite[Sect. 2]{ChengMorris:2003} the transfer function $G(s)$ of (\ref{boundary}) is well-posed and obtained 
by applying the Laplace transform to (\ref{boundary}),  where every function evaluation $G(s)$ requires solving  an elliptic boundary control problem 
\begin{eqnarray}
\label{elliptic}
\Gamma_s:\left\{
\begin{matrix}
sx(s) &\!\!\!=\!\!\!& Ax(s) \\
Px(s) &\!\!\!=\!\!\!& u(s)\\
y(s) &\!\!\!=\!\!\!& Cx(s)
\end{matrix}
\right.
\end{eqnarray}
Well-posedness means that $G(s)$ is holomorphic on a half-plane Re$(s) > \sigma$, but it may be convenient
to require a little more, namely, that $G(s)$ extends meromorphically over a domain containing $\overline{\mathbb C}^+$. This is satisfied in
all cases of practical interest, and guaranteed theoretically
 e.g.  when $G$ is exponentially input/output stabilizable, see \cite[Lemma 8.2.9 (i)(b), (ii)]{Staffans_book}.
The meromorphic form of the transfer function is a necessary requirement 
for applicability of the Nyquist stability test.

After embedding $G(s)$ in a plant $P(s)$ with one or several  closed-loop performance and robustness channels $T_{wz}(P,K)$, we set up
the infinite-dimensional $H_\infty$-optimization problem
\begin{eqnarray}
\label{program}
\begin{array}{ll}
\mbox{minimize} &\|T_{wz}(P,K)\|_\infty \\
\mbox{subject to} & K \mbox{ stabilizes } G \mbox{ in closed loop} \\
& K \in \mathscr K
\end{array}
\end{eqnarray}
where $\mathscr K$ represents a  suitably chosen class of structured controllers with $m$ inputs and $p$ outputs.  In this work we understand the term {\em structured} 
in the sense that controllers $K(\x,s)$
depend differentiably on  a vector $\x\in \mathbb R^n$ of tunable parameters, and have well-posed transfer functions $K_{ij}(\x,s)$, typically with quasi-polynomial
numerators and denominators. Such control laws combine real-rational elements with input and output delays, and can therefore
be physically implemented. In the optimization procedure it will also be necessary to know the finite number
of rhp poles of $K(\x)$  for every $\x$.

Since (\ref{program}) as a rule cannot be solved exactly, we use an inexact 
bundle trust-region method as in \cite{noll:2013,ANR:2015,ANR:17}, 
which guarantees stability of the closed loop, and approximates $H_\infty$-performance up to
a user specified  precision. The following scheme presents our method
in a more formal way.

\begin{algorithm}[ht]
\caption{$H_\infty$-control of infinite-dimensional systems \label{algo1}}
\begin{algo}
\PARAMETERS $\vartheta >0$.
\STEP{Prepare}
Linearize system about steady-state and pre-compute open-loop  transfer function $G(s)$.% of linearized system.
\STEP{Initialize} 
Choose controller structure and
find initial closed-loop stabilizing controller $K(\x^0)$ of that structure.
Let $G_0 = {\tt feedback}(G,K(\x^0))$.
\STEP{Plant}
Embed $G_0$ into plant  $P$ representing desired closed-loop performance specifications.
\STEP{Non-smooth optimization} 
Run inexact bundle trust-region method \cite{noll:2013} with starting point $\x^0$, discretizing (\ref{program}) at each iterate $\x^j$  so that
Nyquist  test guarantees stability of the loop, and $H_\infty$-performance
up to tolerance $\vartheta$.
\end{algo}
\end{algorithm}
 steps of this scheme require further explanations, which we provide in the following sections.

\section{Stability test}
\label{sect_stability}
Let us recall that with
the definitions
\[
F(s) = \begin{bmatrix} I & G(s) \\ -K(s) & I \end{bmatrix}, \qquad f(s) = \det F(s)
\]
the inverse $T(s) = F(s)^{-1}$ is given as
\begin{equation}
\label{T}
T = \begin{bmatrix}
\left( I+KG  \right)^{-1} & -K\left(  I+GK\right)^{-1}\\ \left(I+GK\right)^{-1} G & \left(I+GK\right)^{-1}
\end{bmatrix},
\end{equation}
and we call the closed-loop system $(G,K)$ stable in the $H_\infty$-sense, or simply stable,  if the transfer function
$T$ belong to the Hardy space ${\bf H}_\infty(\mathbb C^+, \mathbb C^{(m+p)\times (m+p)})$. As is well-known, $H_\infty$-stability
is equivalent to the absence of unstable poles in tandem with boundedness of $T(s)$ on $j\mathbb R$. 
We are interested in situations, where absence of unstable poles of $T(s)$ can be verified
by the Nyquist stability test.

Systems arising from parabolic equations are of retarded type and typically satisfy the spectrum decomposition assumption, 
which means that they have only a finite number  of unstable poles. The Nyquist stability test may therefore be applied directly,
to the effect that  in order to guarantee absence of unstable poles in the loop $f(j\omega) = \det (I + G(j\omega)K(\x^0,j\omega))$ has to wind $n_p$ times
around the origin in the clockwise sense,  where $n_{p}$
the number of rhp poles of $G$ and $K(\x)$ together.

In order to address the case where $f$ has a finite number of poles on $j\mathbb R$, 
we consider the following construction, which avoids the usual $\epsilon$-indentation of the Nyquist contour into the rhp.
We choose a holomorphic function $h$ on a domain containing $\overline{\mathbb C}^+$ such that $h(s) \not=0$ on $\mathbb C^+$, $\lim_{s\to \infty} h(s) = 1$ on $\overline{\mathbb C}^+$,
and such that $h$ has a zero of order $p$ at $\pm j\omega$ precisely when $F$ has a pole of order $p$ at $\pm j\omega$. 
Let $\widetilde{f} = fh$,  $\afunc{D}$
a Nyquist D-contour into the rhp with  $[-j\overline{\omega},j\overline{\omega}] \subset \afunc{D}$   containing in its interior all rhp poles of $F$. Then the modified Nyquist curve
$\widetilde{f} \circ  \afunc{D}$ has the same winding number as the original Nyquist curve  $f \circ  \afunc{D}_\epsilon$ with sufficiently small $\epsilon$-indentations.

From the moment onward a controller $K(\x)$ has been identified closed-loop stabilizing using the Nyquist test,
the nonsmooth optimization method, when considering a trial step $K(\x + d\x)$ away from the current iterate $\x$, will re-compute
the winding number to check stability of the the loop with $K(\x+d\x)$. In those cases where
the number of poles of $K(\x)$ is independent of $\x$, this means
we simply have to assure that the winding number $n_p$ does not change as we go from $\x$ to $\x+d\x$, which requires preventing
the
Nyquist curve from crossing the origin. Stability in the $H_\infty$-sense for the new $K(\x+d\x)$  will then follow under the proviso
that the closed-loop transfer function $T(\x+d\x)$ remains bounded on $j\mathbb R$, which is the case
when $G,K(\x)$ are bounded on $j\mathbb R$, and occurs in particular if these transfer functions are proper. With
these preparations the situation for parabolic systems is covered by the following:

\begin{theorem}
\label{theorem1}
{\em
Suppose process $G(s)$ and controller $K(s)$ are well-posed, extend meromorphically into a domain containing $\overline{\mathbb C}^+$, and satisfy the following conditions:
\begin{enumerate}
\item[i.] $F$ has no zeros on $j\mathbb R$ and  only a finite number $N_p$ of poles in $\overline{\mathbb C}^+$, $n_p$ of which are in $\mathbb C^+$.
\item[ii.] There are no pole/zero cancellations on $j\mathbb R$.
\item[iii.] There exist a frequency $\overline{\omega} > 0$ and $\alpha > 0$ such that {\rm Re}$f(j\omega) > \alpha$ for all $\omega \in [\overline{\omega},\infty)$. 
\item[iv.] $G,K$ are  bounded on $j\mathbb R\setminus [-j\overline{\omega},j\overline{\omega}]$. 
\item[v.] $G,K$ have strongly (exponentially) stabilizable and detectable state-space realizations.
\end{enumerate}
Suppose the modified Nyquist curve
$\widetilde{f} \circ  \afunc{D}$ winds $n_p$ times around the origin in the clockwise sense. Then the closed-loop $T(s)$ is strongly (exponentially) stable.
}\hfill $\square$
\end{theorem}

The theorem was proved in \cite{AN:18} under the more standard assumptions that $G,K$ are proper and $\lim_{s\to \infty} f(s) \not=0$ 
on $\overline{\mathbb C}^+$ exists, but since 
our present statement concerns only a finite Nyquist contour, $\widetilde{f}\circ \afunc{D}$, the proof can be adapted with minor changes. 

The full force of Theorem \ref{theorem1} is needed when it comes to dealing with hyperbolic systems.
Here the
situation is complicated because in open loop neutral systems may have infinitely many  rhp poles in a strip $0 \leq {\rm Re}(s) < \alpha$, so 
condition (iii) may fail,
even though  $G$ may still be well-posed. When this is the case, it is
impossible to use the Nyquist test directly even if
an initial stabilizing controller $K_0=K(\x^0)$ is given, because 
the Nyquist curve $f(s)=\det(I+G(s)K_0(s))$ winds infinitely
often around the origin. In this event the method explained in the following section is helpful. 

\subsection{Enabling the Nyquist test}
\label{G0}
Assuming that is has been verified by some other means that the closed-loop system
$G_0= G(I+K(x_0) G)^{-1}=: {\tt feedback}(G, K(\x^0))$ is indeed stable. Then a small gain argument
tells us that ${\tt feedback}(G_0,K)$ remains stable for stable controllers $K$ satisfying $\|K\|_\infty < 1/\|G_0\|_\infty$. Since
${\tt feedback}(G_0,K) = $ $
{\tt feedback}({\tt feedback}(G,K(\x^0)),K) $ $ = {\tt feedback}(G,K(\x^0)+K)$, we see by letting $K$  the difference $K(\x)-K(\x^0)$ that
${\tt feedback}(G,K(\x))$ is stable for $\|K(\x)-K(\x^0)\|_\infty < 1/\| G_0\|$, and this can now be verified by applying the Nyquist test
to  $f_0(s) = \det (I+ G_0(s)K(s))$, where $K = K(\x)-K(\x^0)$. 

Here Theorem \ref{theorem1} applies indeed to $G_0$, $K$, because $G_0$ has no poles on $\overline{\mathbb C}^+$, so that i.-iv. are satisfied for $G_0,K$
provided  iii. was from start satisfied for $G,K_0$. We have, however, to recall that despite stability of the loop,  $f_0$ will typically not have a limit 
as $s \to \infty$ on $\overline{\mathbb C}^+$, so we will rely on condition iv., which assures that outside the band $[-\overline{\omega},\overline{\omega}]$
the Nyquist curve is in no danger to turning around 0.

From here on, we can proceed just as in the previous case for retarded systems,
where now $N_p=n_p$ is the known number of rhp poles of $K=K(\x)-K(\x^0)$. 
%Since this requires the difference $K(\x)-K(\x^0)$ to be stable, this argument is in practice confined to stable controller structures $K(\x)$.

\begin{remark}
Should our initial stabilizing controller $K_0$ be unstable, it may be preferable to
use controllers of the form $K_0 + K(\x)$ for $K(\x)$ stable, which gives rise to a modified structure. 
\end{remark}

The question remains in what sense a state-space representation of $T$ given by
(\ref{T}) will be stable. This is decided by the 
following

\begin{theorem}
\label{theorem2}
Suppose $G,K$ are well-posed transfer functions which admit strongly (exponentially) stabilizable
state-space realizations. Suppose the closed loop transfer function satisfies
$T \in {\bf H}_\infty$. Then the generator of the state-space representation of the closed loop
is strongly (exponentially) stable. 
\end{theorem}

\begin{proof}
For exponential stability,
according to Morris \cite[Theorem 5.2]{Morris:1999} it suffices to show that the closed
loop is exponentially stabilizable and exponentially detectable. 
By Staffans \cite[Lemma 8.2.7]{Staffans_book} this follows as soon as each of the components
$G,K$ is individually exponentially stabilizable and detectable, but the latter is true by hypothesis.

For the statement concerning strong stability we use \cite[Lemma 8.2.7]{Staffans_book} again, which now guarantees
that the closed loop $T(G,K)$ is strongly stabilizable and detectable. Since it is $H_\infty$-stable, we can invoke
\cite[Theorem 8.2.11 (ii)]{Staffans_book} to infer that the closed loop is also strongly stable.
\end{proof}

\begin{figure}[h]
\includegraphics[scale=1.1]{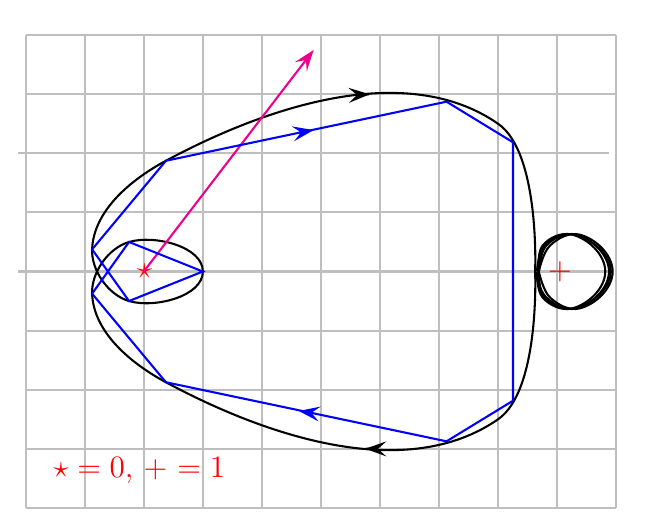}
%\begin{figure}[h]

\caption{Winding number of Nyquist curve (black) and polygon (blue)  agree if change of argument of segment
$[\widetilde{f}(j\omega_\nu),\widetilde{f}(j\omega_{\nu+1})]$ equals change of argument of corresponding piece of Nyquist curve $\widetilde{f}([\omega_\nu,\omega_{\nu+1}])$, as guaranteed by condition (\ref{L}). Winding number of polygon
$P_{\widetilde{f}}$  is computed by counting ray crossings in red. \label{fig_nyquist}}
\end{figure}

%%%%%%%%%%%%%%%%%%%
%%%%%%%%%%%%%%%%%%%
%%%%%%%%%%%%%%%%%%%

\subsection{Sampling for the Nyquist test}
\label{nyquist}
During optimization,
the Nyquist test is applied at every candidate controller $K(\x)$ to check stability of the loop ${\tt feedback}(G,K(\x))$, respectively for neutral systems,  of
the preliminary stabilized system
${\tt feedback}(G_0,K(\x)-K(\x^0))$.
For that we have to sample the modified Nyquist curve $\widetilde{f}=fh$     %, based on  \cite{AN:18,LesChinois:1993}
at frequencies $0 =\omega_0  < \dots < \omega_N = \overline{\omega}$ such that the winding
number of $\widetilde{f} \circ \afunc{D}$, based on a typical D-contour extending into the rhp, coincides with the winding number of the closed polygon $P_{\widetilde{f}}$
obtained as
$\widetilde{f}(j\omega_0),\dots, \widetilde{f}(j\omega_N)$, $\widetilde{f}(-j\omega_N)$, $\widetilde{f}(-j\omega_{N-1}),\dots,\widetilde{f}(-j\omega_0)$. 
The latter can then be computed conveniently using the ray crossing algorithm (see Fig. \ref{fig_nyquist})

To begin with, we need a cutoff frequency $\overline{\omega} > 0$ such that
every Nyquist contour
$\afunc{D}$ containing the segment $[-j\overline{\omega},j\overline{\omega}]$ contains all $n_p$ unstable poles of $G,K$ in its interior.
We then sample on the segment $[0,\overline{\omega}]$, and replace the 
curved part of the D-contour   $\widetilde{f}\, \circ $\,\rotatebox[origin=c]{180}{$\afunc{C}$}   by
the segment $[\widetilde{f}(-j \overline{\omega}),\widetilde{f}(j \overline{\omega})]$. For the latter to be authorized, 
the closed curve $\gamma$ concatenated by this segment and  $\widetilde{f}\, \circ $\,\rotatebox[origin=c]{180}{$\afunc{C}$}  has to satisfy
${\rm ind}(\gamma,0)=0$ and should not contain any of the $n_p$ poles
of $F$ in its interior. This is for instance the case if we choose $\overline{\omega}$ as
in condition iv. of Theorem \ref{theorem1}.
Then we sample $\widetilde{f}$ on $[0,\overline{\omega}]$
such that none of the curves $\gamma_i$ obtained by concatenating the segment $[\widetilde{f}(j\omega_{i+1}),\widetilde{f}(j\omega_{i})]$ with the piece $\widetilde{f}([j\omega_i,j\omega_{i+1}])$
of the modified Nyquist curve encircles $0$. This can be assured my the method in \cite{AN:18}. 
Call a mapping $L[\cdot,\cdot]:\mathbb R^2\to \mathbb R$ a first-order bound of $\widetilde{f}$
if $|\widetilde{f}'(j\omega)| \leq L[\omega^-,\omega^+]$ for all $\omega\in [\omega^-,\omega^+]$. Then we have the following

\begin{lemma}{\em 
Suppose that for fixed $K$ the cutoff frequency $\overline{\omega}>0$ is as above, and the sampling nodes satisfy 
the condition
\begin{equation}
\label{L}
L[\omega_i,\omega_{i+1}] (\omega_{i+1}-\omega_i) < |\widetilde{f}(j\omega_i)| + |\widetilde{f}(j\omega_{i+1})|.
\end{equation}
Then the winding numbers of the modified Nyquist curve  $\widetilde{f} \circ \afunc{D}$    and the approximating closed polygon $P_{\widetilde{f}}$
agree.
%\hfill $\square$
}
\end{lemma}

For the proof see \cite[sect 3.]{AN:18}.   \hfill $\square$

\section{Sampling for performance}
\label{barrier}
Sampling for $H_\infty$-performance was also analyzed in \cite{AN:18} and can again be based on a first-order bound $L[\cdot,\cdot]$, now for
the function $\phi(\omega) = \overline{\sigma} \left( T_{wz}(P(j\omega),K(j\omega))  \right)$. We recall the following
\begin{lemma}
\label{lem1}
{\em
Let $\gamma^* = \max\{\phi(\omega_i): i=1,\dots,N\}$ for a given controller $K$ and a corresponding sampling $\omega_i$. 
Let $\vartheta > 0$ be a user-specified tolerance. If the nodes $\omega_i$ satisfy
\begin{equation}
\label{LL}
L[\omega_i,\omega_{i+1}] (\omega_{i+1}-\omega_i) < 2 \gamma^* + 2 \vartheta - \phi(\omega_i)-\phi(\omega_{i+1}),
\end{equation}
 then the true
$H_\infty$ norm is within tolerance $\vartheta$ of its estimated value, that is,
$\gamma^* \leq \|T_{wz}(P,K)\|_\infty \leq \gamma^* + \vartheta$.
\hfill $\square$
}
\end{lemma}

For the proof see \cite[sect. 5]{AN:18}.
In that work we have compared 
sampling for stability via (\ref{L}) and sampling for $H_\infty$-performance based on (\ref{LL}) on a large test bench including
finite and infinite dimensional systems. The results fairly consistently show that performance requires at least 10 times more nodes $\omega_i$
than sampling to assure that the Nyquist test is correct. This leads to the following significative
meta-theorem: {\em $H_\infty$-performance is 10 times more costly than mere stability}. 

In those cases where a channel $T=T_{wz}(P,K)$ for $H_2$-optimization is available, we need a sampling
$\omega_i$ for $f(\omega) = {\rm trace}\left(T(j\omega)^H T(j\omega) \right)$. Assume that a first-order bound $L[\omega^-,\omega^+]$ for $f$
is available, and let $P$ be the piecewise linear function corresponding to the polygon with nodes $(\omega_i,f(\omega_i))$,
$0 =\omega_0 < \dots < \omega_N=\overline{\omega}$ with $P(\omega)=0$ for $\omega > \overline{\omega}$. 

\begin{lemma}
{\em 
Let $\vartheta > 0$ be a user specified  tolerance and suppose the cutoff frequency $\overline{\omega} >0$ is such that
\[
e_1=\int_{\overline{\omega}} ^\infty f(\omega) d\omega  \leq \vartheta/2.
\]
Suppose the interval $[0,\overline{\omega}]$ is sampled with nodes $\omega_i$ such that
\begin{equation}
\label{LLL}
\textstyle\frac{\overline{\omega}}{4} (\omega_{i+1}-\omega_i) L[\omega_i,\omega_{i+1} ] \leq \vartheta/2,
\end{equation}
then the error satisfies
\[
e=\left| \int_0^\infty f(\omega)d\omega - \int_0^\infty P(\omega) d\omega \right| < \vartheta.
\]
}
\end{lemma}
\begin{proof}
Let $e_1$ be the error of the high frequency contribution satisfying $e_1 < \vartheta/2$. Now the error
of the low frequency part is
$e_1 = \left| \int_0^{\overline{\omega}} f(\omega) d\omega - \int_0^{\overline{\omega}} P(\omega) d\omega \right| 
\leq \sum_{i=0}^{N-1} \frac{1}{4} (\omega_{i+1}-\omega_i)^2 L[\omega_i,\omega_{i+1}]
\leq \sum_{i=0}^{N-1} (\omega_{i+1}-\omega_i) \vartheta	 /2\overline{\omega} = \vartheta/2.
$ Hence altogether $e = e_1+e_2 < \vartheta$.
%\hfill $\square$
\end{proof}

\begin{remark}
It is clear that (\ref{LLL}) is much more binding than (\ref{LL}), because sampling to assure
the exactness of the maximum value within a tolerance $\vartheta > 0$ has only to be precise at frequencies close to the maximum, 
whereas approximating an
integral requires good approximation on the whole $[0,\overline{\omega}]$. This suggests avoiding
$H_2$-optimization if possible. Since robustness requirements further press to avoid $H_2$-optimization, we presently 
seek for workarounds.
\end{remark}

\section{Optimization}
With a computable test for stability   and a method to approximate the objective function $\|T_{wz}\|_\infty$ available,
we run our nonsmooth optimization method based on \cite{AN:18,ANR:2015},
with the interpretation of  inexact function and subgradient  evaluations as in \cite{noll:2013}.
For this we have to recall subgradient evaluation as discussed in \cite{AN06a}. Suppose $\omega_i$ is one of the sample  frequencies
where the maximum $\gamma^* = \phi(\omega_i)$ of the approximation is attained with error $\gamma^* \leq \|T_{wz}(\x)\|_\infty \leq \gamma^* + \vartheta$. Then
an approximate subgradient is generated by computing one or several maximum eigenvectors $\Phi_i$ of $T_{wz}(\x,\omega_i)^HT_{wz}(\x,\omega_i)$ 
and using formulas (12) or (13) of \cite[Sect. IV]{AN06a}. Inspecting those shows that for the nearby frequency
$\omega$ where $\|T_{wz}(\x)\|_\infty$ is in reality attained, the mismatch between the estimated eigenvector $\Phi_i$ and one of the true maximum eigenvectors 
$\Phi$  of the $H_\infty$-norm at
$\omega$ is proportional to the error in the function values, but with proportionality constant 
depending on the reciprocal of the eigenvalue gap at $T_{wz}(\x,\omega)$. Since only finitely many frequencies are active, the eigenvalue gap cannot
become arbitrarily small.  From the same reason it remains bounded in the neighborhood of any of the accumulation points of the sequence of serious iterates generated by
the bundle or the bundle trust-region method. That suggests indeed an interpretation of our method as an instance of the inexact bundle trust-region method.

For a recent thorough convergence analysis of the bundle method with inexact function and subgradient evaluations 
in an infinite-dimensional setting we refer to Hertlein and Ulbrich \cite{hertlein_ulbrich:18}.

Starting the algorithm as presented in \cite{AN:18,ANR:2015} at the closed-loop stabilizing $K_0=K(\x^0)$, the Nyquist test is used at every new iterate
$K(\x^+)$ to check whether the loop ${\tt feedback}(G,K(\x^+))$, respectively ${\tt feedback}(G_0,K(\x^+)-K(\x))$ for the case of neutral systems,
is stable. If this is not the case, a backtracking step $\x_\alpha =  \x + \alpha(\x^+-\x)$ for $0 < \alpha < 1$ is made such that
$K(\x_\alpha)$ is still stabilizing, and a repelling cutting plane
is included in the bundle, using e.g. the closed-loop sensitivity function
$S(\x) = \|( I+G K(\x))^{-1}\|_\infty$ as a stability barrier function. See \cite[sect. 4]{AN:18}  for details.

\begin{remark}
A special situation occurs if $G,K$ are stable, or if $G_0$ instead of $G$ is used as described in section \ref{G0}.
Here the Nyquist curve $f(j \omega)$ does not at all wind around the origin, and as a rule stays outside a conical or parabolic region
$R_{\alpha,r}=\{s \in \mathbb C: {\rm Im}(s)^2 < \alpha (r-{\rm Re}(s))\}$ for certain $\alpha,r > 0$. The constraint ${\rm Im}(f(\x,j\omega))^2 \geq  \alpha (r-{\rm Re}(f(\x,j\omega)))$
can be easily included in the constrained program (\ref{program}).
This is more reliable than just 
preventing $f(\x,j\omega)$ from  crossing the origin. Note that even when the Nyquist test can be applied directly to $G,K(\x)$, e.g. for parabolic
or first-order hyperbolic systems, it may for numerical reasons be
interesting to apply it in just the same way to $G_0,K$.

A more conservative but robust way to address the same problem is to add a disk-margin constraint of the form $\|S\|_\infty \leq 1/\alpha$. 
\end{remark}

\section{Output-feedback control of a reaction-diffusion equation with input delay}
\label{parabolic}
In this section we discuss a one-dimensional reaction diffusion equation
with delayed Dirichlet boundary control
\begin{align}
\label{diff}
&x_t(\xi,t) = x_{\xi\xi}(\xi,t) + c(\xi)x(\xi,t), \quad t \geq 0, \xi \in [0,L], \notag \\
&x(0,t) = 0,\\
&x(L,t) = u(t-D)\notag
\end{align}
where $x(\cdot,t)$ denotes the state of the system, $u(t)$ the control,
$D$ the delay,
and where we assume that a finite number of measured outputs
$y_1(t)=x(\xi_1,t),\dots,y_m(t)=x(\xi_m,t)$ at sensor positions
$\xi_i\in [0,L]$  are available for control.  A similar
control  scenario is discussed in Prieur and Tr\'elat \cite{prieur_trelat:17} under the assumption of full state measurement.
Related work is for instance Sano \cite{Sano2017}, where $H_\infty$-control of a heat exchanger
is discussed, or
\cite{ANR:2015}, where a reaction-convection-diffusion 
equation with simultaneous boundary and distributed
control and a van de Vusse reactor of a coupled system of
reaction convection-diffusion equations again with combined
boundary and distributed control are discussed without input delay, but with a single
point measurement as output. 

In the present study, 
%we mimick a practical situation, where we have only access to a finite 
%number of point measurements assumed for simplicity
%regularly spaced along the $\xi$-axis. Moreover, 
we strive to control the system with a finite-dimensional output
feedback controller $K(\x)$ of simple structure, which could conveniently be implemented,
and yet gives satisfactory performances in closed loop.

Performance specifications of the reaction-diffusion PDE 
are chosen so that responses to non-zero initial conditions show
reasonable behavior in terms of damping and settling time. This could be addressed 
by an $H_2$-performance specification, but as we show may also
be successfully controlled by way of suitably chosen $H_\infty$-specification.
The latter is advantageous as soon as additional robustness aspects of the design are
called for.

Working incrementally, and starting with the case of a single
measurement at the mid-point $\xi = L/2$,
our analysis indicates that $5$ equidistant measurements are 
enough to achieve good responses against initial conditions, while mere stability
could be assured even on the basis of a single measurement e.g. at $\xi=0$.

In our numerical testing we adopt the choices $L = 2\pi$, $D=1$ and
$c(x) = \frac{1}{2}$ from \cite{prieur_trelat:17}, where the open-loop $F(s)$
has one unstable pole at $s=\frac{1}{2}$, and an infinity of stable double poles at $s_k = \frac{1}{2} - \frac{k^2\pi^2}{L^2}$
following a retarded pattern. This is understood, as the
semi-group of the equation is sectorial \cite[p.~150]{Staffans_book}.  As indicated
in previous chapters, this allows direct application of the Nyquist test
to check $H_\infty$-stability of the closed loop.

\begin{remark}
Well-posedness of the system (\ref{diff}) in the sense of Salamon-Weiss can be deduced from the functional analytic setting
in \cite{prieur_trelat:17}, or from the general approach in \cite{ChengMorris:2003}. 
\end{remark}

%\begin{remark}
%The fact that the semi-group satisfies the spectrum decomposition property
%is beneficial for theoretical aspects, but not for practical synthesis, so
%we steer clear of the very common error to rely on the spectrum decomposition
%to design controllers.
%\end{remark}

The transfer function of (\ref{diff}) can be computed analytically as
\[
G(s,\xi) = \frac{x(\xi,s)}{u(s)} = e^{-s} \frac{e^{\sqrt{s-\frac{1}{2}}\xi }- e^{-\sqrt{s-\frac{1}{2}} \xi}}{e^{\sqrt{s-\frac{1}{2}}L} - e^{-\sqrt{s-\frac{1}{2}}L}},
\]
and for $L=2\pi$ the system has one unstable pole. 

\begin{proposition}  
\label{prop1}
{\em
Suppose a finite-dimensional structured controller $K(\x)$ with $m$ inputs and $p=1$ output is found which stabilizes 
system} (\ref{diff}) {\em internally in the $H_\infty$-sense. Then the closed loop
is even exponentially stable.}
%\hfill $\square$
\end{proposition}

\begin{proof}
By theorem \ref{theorem2} this 
follows as soon as each of the components
$G,K$ is individually exponentially stabilizable and detectable. Since $K$ is
finite-dimensional it has clearly  an exponentially stabilizable and detectable
state-space model.  For $G$ exponential stabilizability may be deduced from 
\cite{prieur_trelat:17}, because the infinite-dimensional state-feedback controller
the authors construct has the same control input as (\ref{diff}). Exponential 
detectability on the other hand follows from the fact that the differential operator in (\ref{diff})
is self adjoint, and that in the adjoint system the five outputs are turned into 5 inputs, one of which
is the same as the single input in (\ref{diff}), but now without the delay. 
Exponential detectability therefore follows from the fact that
(\ref{diff}) is exponentially stabilizable without the delay.
\end{proof}

\vspace{.1cm}
This still leaves the problem of {\em finding}
a preliminary stabilizing controller $K_0=K(\x^0)$ of the pre-defined structure. 
As we indicated in previous sections, the latter as a rule
requires heuristic methods even for  very simple structures. 
The advantage we have in the case of the present parabolic study (\ref{diff}) is that we can check
via the Nyquist test whether a given controller is stabilizing.

%In the present study we will discuss two control  strategies, which compute the preliminary stabilizing $K_0$ individually.
%They are discussed in turn in the following. 

\subsection{Model matching approach}
Model matching is a sophisticated control scenario, where specifications are pursued indirectly. It is covered by algorithm \ref{algo1}, and
we believe it is particularly suited for PDE-control, where models of different grid scales arise naturally. Here we use model matching to address the
reaction of the system to a non-zero initial value. 
The method consists in two steps (a) and (b).

Non-vanishing initial values may be regarded as disturbances $d$ acting on the system state as in Fig. \ref{red_syn}.  We now assume that we have to
regulate against functions $x_0(\xi)\not=0$ of a given bandwidth. In a
first step (a) we therefore compute a reduced state-space model $G_{\rm red}(s)$ of $G(s)$, in which the resolution of the state $x_r(\xi)$ 
reflects the resolution of the potential $x_0(\xi)$ accurately. In the present study we regulate against  initial conditions with resolution comparable to
that considered in 
\cite{prieur_trelat:17}, which leads  us to a finite-difference discretization of (\ref{diff}) with $50$ spatial steps
or states, 
complemented by a $3$rd-order 
Pad\'e approximation of the input delay adding $3$ more states.  
This coarse grid model $G_{\rm red}$ 
is embedded into a plant $P_{\rm red}$  expressing control requirements in terms of damping and settling time in responses to initial conditions.
Here this consists in optimizing the
root mean-square energy value of the output signal $z_r=(z_1,z_2)$ in response to the white noise disturbance $d$ on the state $x_r$:
%The finite-difference discretization of (\ref{diff}) leading to $G_{\rm red}(s)$ is embedded into the plant $P_{\rm red}$
\begin{eqnarray*}
P_{\rm red}(s) :
\left\{\;
\begin{matrix}
\dot{x}_r &= &A_r x_r &+& d &+& B_2 u_r \\
z_r &=& C_1 x_r &&&+& D_{12} u_r \\
y_r &= &C_2 x_r& +& D_{21} d 
\end{matrix}
\right.
\end{eqnarray*}
where $x_r\in \mathbb R^{54}$  is the reduced state,  $d$ is the exogenous input, understood to represent the impulse caused by the non-zero initial value, 
and $z_r=(W_x x_r,W_uu)$  is regulated similar to what is used in LQG-control, with filters $W_x=I$ and $W_u(s)=\frac{s}{1+s/a}$, $a=100$, the latter adding another state to $x_r$. 
The reduced output $y_r\in \mathbb R^5$ of $G_{\rm red}$ represents the 5 distributed measurements (see Fig. \ref{figResp1}) in the coarse finite-difference
discretization.  Our testing shows 
that five equidistant measurements along $[0,\, L]$ are sufficient to achieve well behaved responses to initial conditions. 

\begin{figure}[ht!]
\includegraphics[scale=1.2]{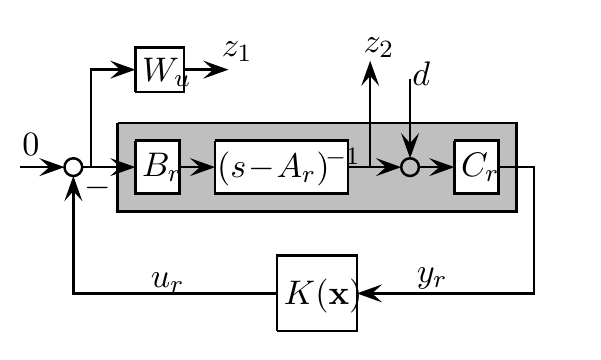}
\caption{Model matching. Preliminary stabilization via structured $H_2$-synthesis using $P_{\rm red}$ with channel
$d \to z_r = (z_1,z_2)=(W_uu_r,x_r)$. \label{red_syn}}
\end{figure}

We then solve the structured $H_2$-optimization
problem for the channel $d\to z_r$ of $P_{\rm red}$ in Fig. \ref{red_syn}, 
using {\it systune} \cite{RCT2013b} based on \cite{AN06a,an_disk:05}. Trial and error
reveals the interesting fact that fairly low-order controllers $K\in \mathscr K_2$ of order 2 are adequate, leading to
$\x \in \mathbb R^{21}$ for $K$:        
\begin{eqnarray}
\label{H2}
\begin{array}{lll}
\mbox{minimize} & \|T_{d \to z_r}(P_{\rm red},K)\|_2 \\
\mbox{subject to} & K(\x) \mbox{ stabilizes } G_{\rm red} \mbox{ internally} \\
&K(\x) \in \mathscr K_2
\end{array}
\end{eqnarray}
where the closed-loop transfer $T_{d \to z_r}(P_{\rm red},K)$ is the lower linear fractional transformation $F_l(P_{\rm red},K)$. 
The solution of (\ref{H2}) is $K_0$ in Figure \ref{figure1} right. 
The difference with traditional $H_2$-control is that $K$ does not have observer structure, but the structure
$K \in \mathscr K_2$ we imposed. 
The resulting controller $K_0\in \mathscr K_2$ is obtained as 
\begin{align}\label{K0}
                K_0^{11}   & =  \frac{ 0.001653 s^2 + 0.822 s + 5.557}{ s^2 + 4.315 s + 18.3}   \notag \\
 K_0^{12}&=
 \frac{
  0.01467 s^2 + 3.125 s + 20.69}
{      s^2 + 4.315 s + 18.3}  \notag \\
 K_0^{13}
&=\frac{0.0221 s^2 + 4.784 s + 31.2}
{     s^2 + 4.315 s + 18.3}\\
K_0^{14}
&=\frac{  0.01733 s^2 + 3.715 s + 24.34}
      {s^2 + 4.315 s + 18.3}\notag \\
K_0^{15}&= \frac{0.00231 s^2 + 0.9017 s + 6.596}
{       s^2 + 4.315 s + 18.3}\notag
\end{align}
A simulation for initial condition $x_0(\xi) = \xi (L-\xi) $ is shown in Fig. \ref{figResp1}  % formerly  \ref{figResp0} 
(left).  While this produces the expected good results for $G_{\rm red}$, 
we now have to
check whether $K_0$ also stabilizes the infinite-dimensional system $G(s)$. It turns out that this is the case, as the Nyquist test reveals, so that we now
proceed to
the second part (b) of the model matching method, where the controller is further optimized with regard to the full model.  

%\begin{figure}[h!]
%\centering
%%%\vspace*{-.5cm}
%\includegraphics[height=5cm]{figMM0.png}
%\caption{Model matching. Response to initial condition of $53$rd-order model $G_{\rm red}$ in closed-loop with $H_2$-controller $K_0\in \mathscr K_2$ in (\ref{K0}).
%The \textcolor{red}{ \mbox{\large$\ast$} } indicate measurements locations.}
%\label{figResp0}
%\end{figure}

%\begin{remark}
%The coarse state-space model $G_{\rm red}$ accounting for the initial condition has to be adapted to
%the spatial resolution of the initial signal $x_0(\xi)$ we expect. We assume that the choice of 50 spatial steps chosen in our example
%allows to represent $x_0(\xi)$ accurately. 
%\%end{remark}

As a result of step (a) of the model matching procedure
we have so far obtained a reference model, ${\tt feedback}(G_{\rm red},K_0)$ and a controller $K_0$ of the desired structure 
which stabilizes $G_{\rm red}$. 
Application of the Nyquist test, in tandem with boundedness of the closed loop transfer function on $j\mathbb R$, show that $K_0$ also stabilizes 
the infinite dimensional system $G(s)$. Fig. \ref{figNyq0} (left) shows one clockwise encirclement, 
computed by the ray-crossing algorithm, 
which due to the known single unstable pole
in $G$, and absence of unstable poles in $K_0$, confirms the absence of unstable closed-loop poles. Taking into account that $G,K_0$ are proper shows
that the closed loop transfer function is bounded on $j\mathbb R$, hence the loop is $H_\infty$-stable, and by Proposition \ref{prop1}, is exponentially stable.

\begin{remark}
The fact that $K_0$ stabilizes not only $G_{\rm red}$, but even $G$, could be called accidental.  However, recall that
within most structures $K \in \mathscr K$, practical methods leading to a stable closed loop are necessarily heuristics, so remain equally accidental. What can be said in favor of 
our method to obtain $K_0$
is that it is the result of a local optimization procedure.
\end{remark}

\begin{figure}[ht!]
\includegraphics[scale=1.2]{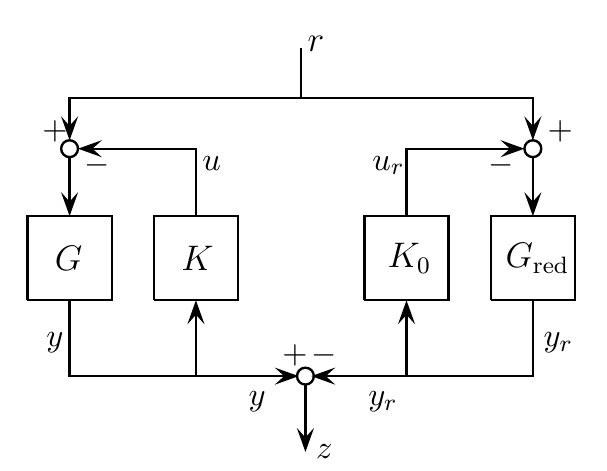}
\caption{Model matching  step (b). $H_\infty$-optimization of channel $r \to y-y_r$ in (\ref{model_matching})
starting at initial guess $K=K_0$ leads to overall optimal $K=K^*$ and is a special case of (\ref{program}). In the reference model on the right $K_0$ 
obtained in step (a) remains now fixed. \label{figure1}}
\end{figure}

In a second step (b) of the model matching procedure, corresponding to step 3 of algorithm \ref{algo1},
the preliminary stabilizing  controller $K_0$ is  refined through an $H_\infty$ model matching problem shown in Fig. \ref{figure1}, 
which takes the true  infinite-dimensional dynamics in $G(s)$
accurately into account. In this step, 
we have to solve an infinite-dimensional structured $H_\infty$-control problem covered by the general form  (\ref{program}), 
\begin{eqnarray}
\label{model_matching}
\begin{array}{ll}
\mbox{minimize} & \|(I+G_{\rm red}K_0)^{-1}G_{\rm red}K_0 - (I+G K)^{-1} G K \|_\infty \\
\mbox{subject to}& K \mbox{ stabilizes } G \mbox{ internally} \\
&K \in \mathscr K_2
\end{array}
\end{eqnarray}
for which we use the bundle algorithm of \cite{ANR:2015,noll:2013,AN:18},
initialized at $K=K_0$.   As a result of optimization the $H_\infty$-norm of the mismatch
channel $r \to z=y-y_r$ is reduced from
$1.81$ at $K_0$ to $0.84$ at $K^*$.

The optimized controller $K^*$ obtained is
\begin{align}
\label{opt1}
K^*_{11}&= \frac{0.1343 s^2 + 0.4535 s + 11.34}{s^2 + 10.66 s + 38.39}\notag \\
 K^*_{12}&=\frac{
  0.52 s^2 + 1.755 s + 45.23}{
 s^2 + 10.66 s + 38.39} \notag \\
 K^*_{13}&= \frac{0.7443 s^2 + 2.621 s + 65.23}{s^2 + 10.66 s + 38.39} \\
 K^*_{14}&= \frac{0.5976 s^2 + 2.036 s + 52.82}{s^2 + 10.66 s + 38.39}\notag \\
K^*_{15}&= \frac{ 0.3446 s^2 + 2.621 s + 20.47}{s^2 + 10.66 s + 38.39} \notag
\end{align}
its simulation is shown in Figs. \ref{figResp1} (right) and the corresponding Nyquist plot  in Fig.  \ref{figNyq0} (right). 
The results resemble  those obtained by studies based on full-state information, see
\cite{prieur_trelat:17}.

\begin{figure}[h!]
\centering
%\vspace*{-.5cm}
\includegraphics[height=5.1cm]{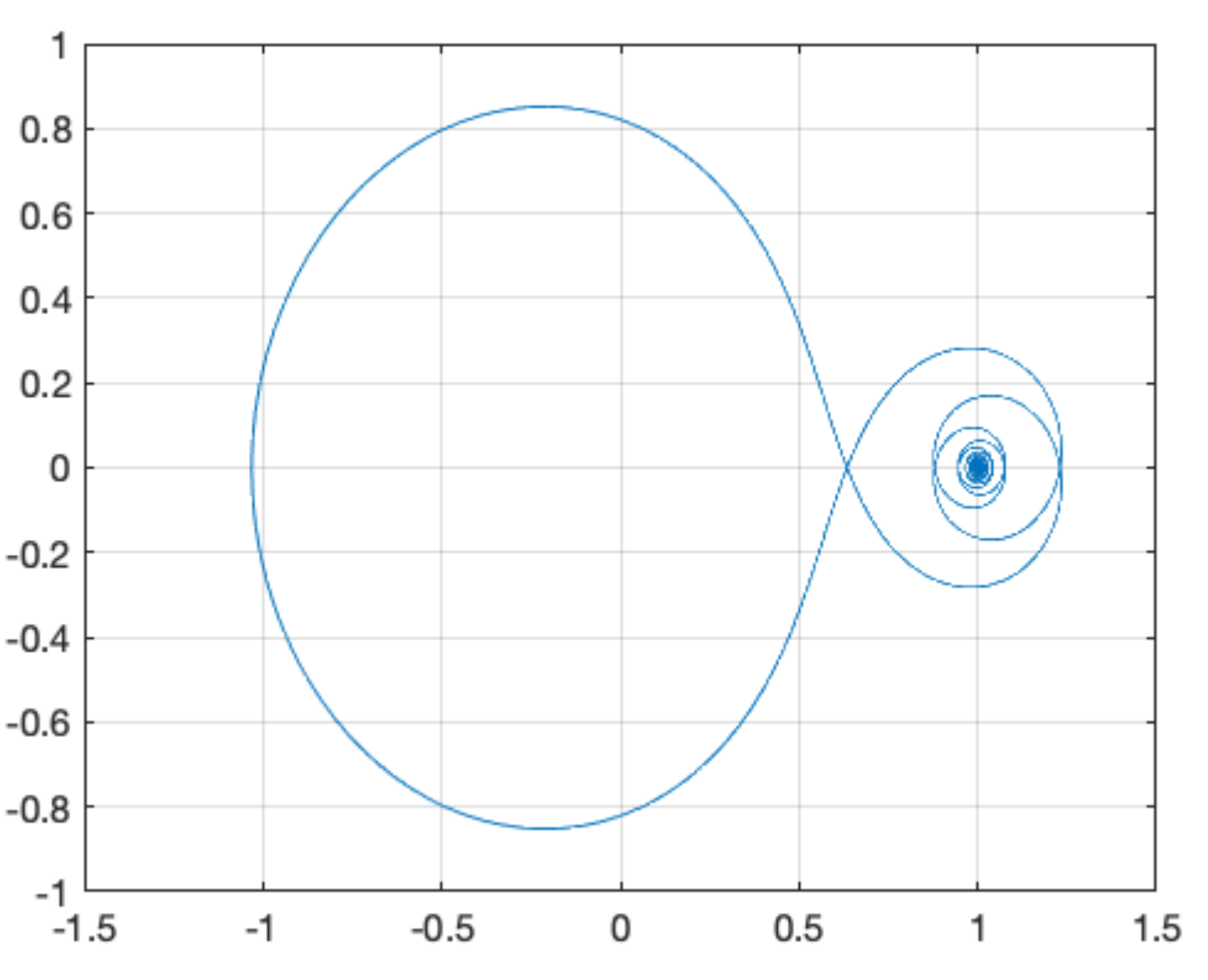}$\!\!\!$
\includegraphics[height=5.1cm]{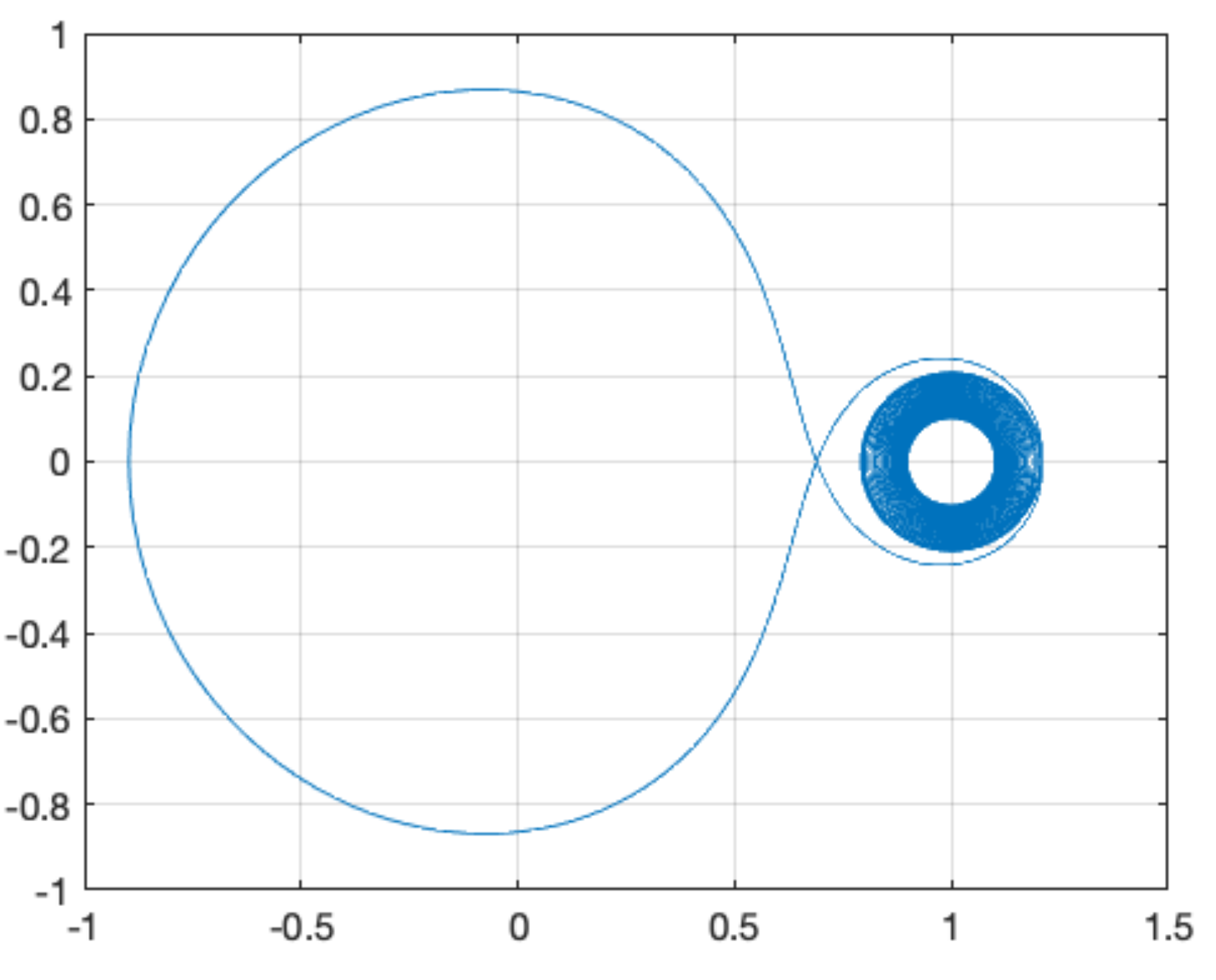}
\caption{Model matching. Nyquist  curve $1+GK_0$ (left) with initial controller (\ref{K0}) and $1+GK^*$ (right).
Since $F(s)$ has one unstable pole,  one counterclockwise encirclement confirms absence of unstable poles
in the loop. \label{figNyq0}}
\end{figure}

\begin{figure}[h!]
\centering
%\vspace*{-.5cm}
\includegraphics[height=5.3cm]{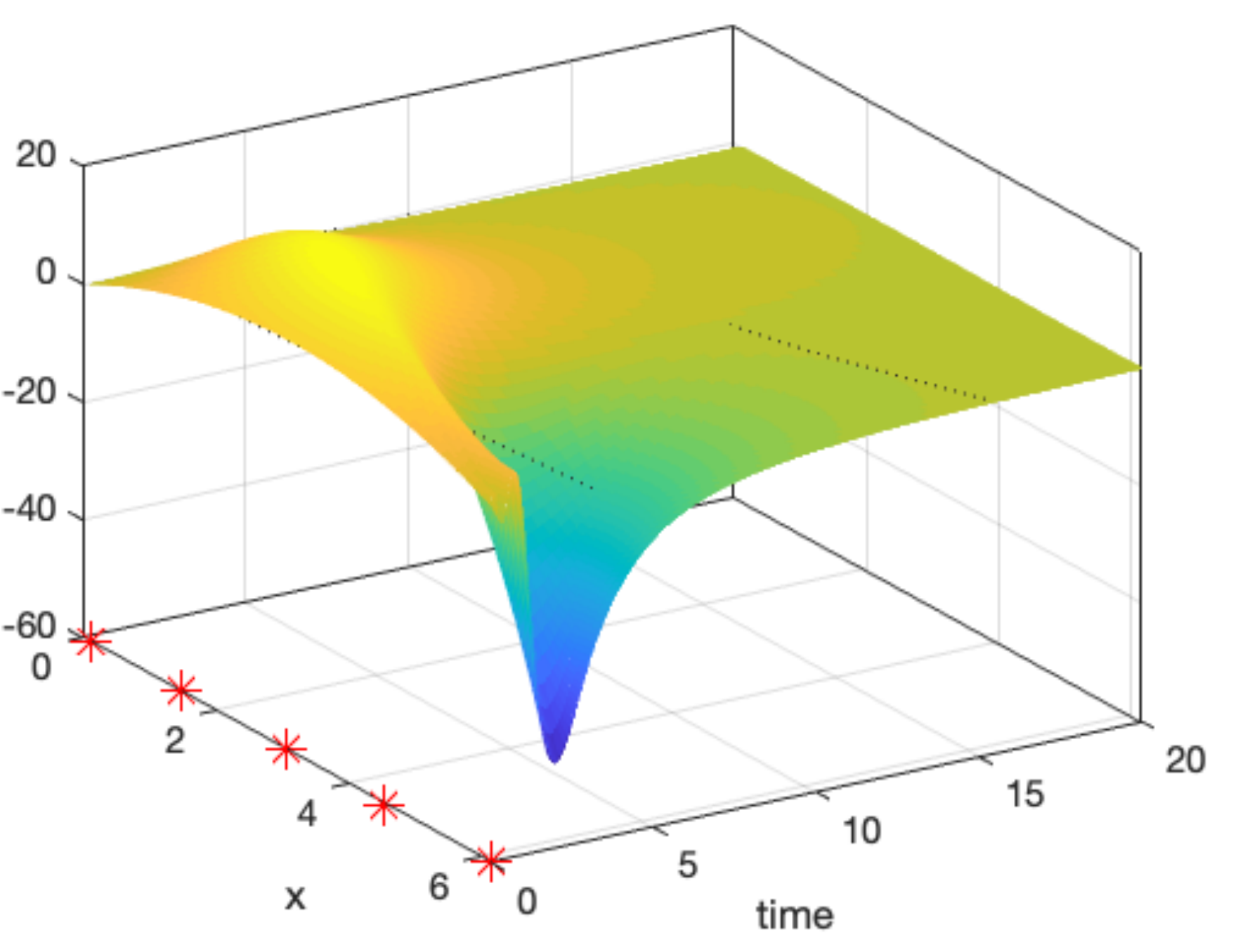}
\includegraphics[height=5.3cm]{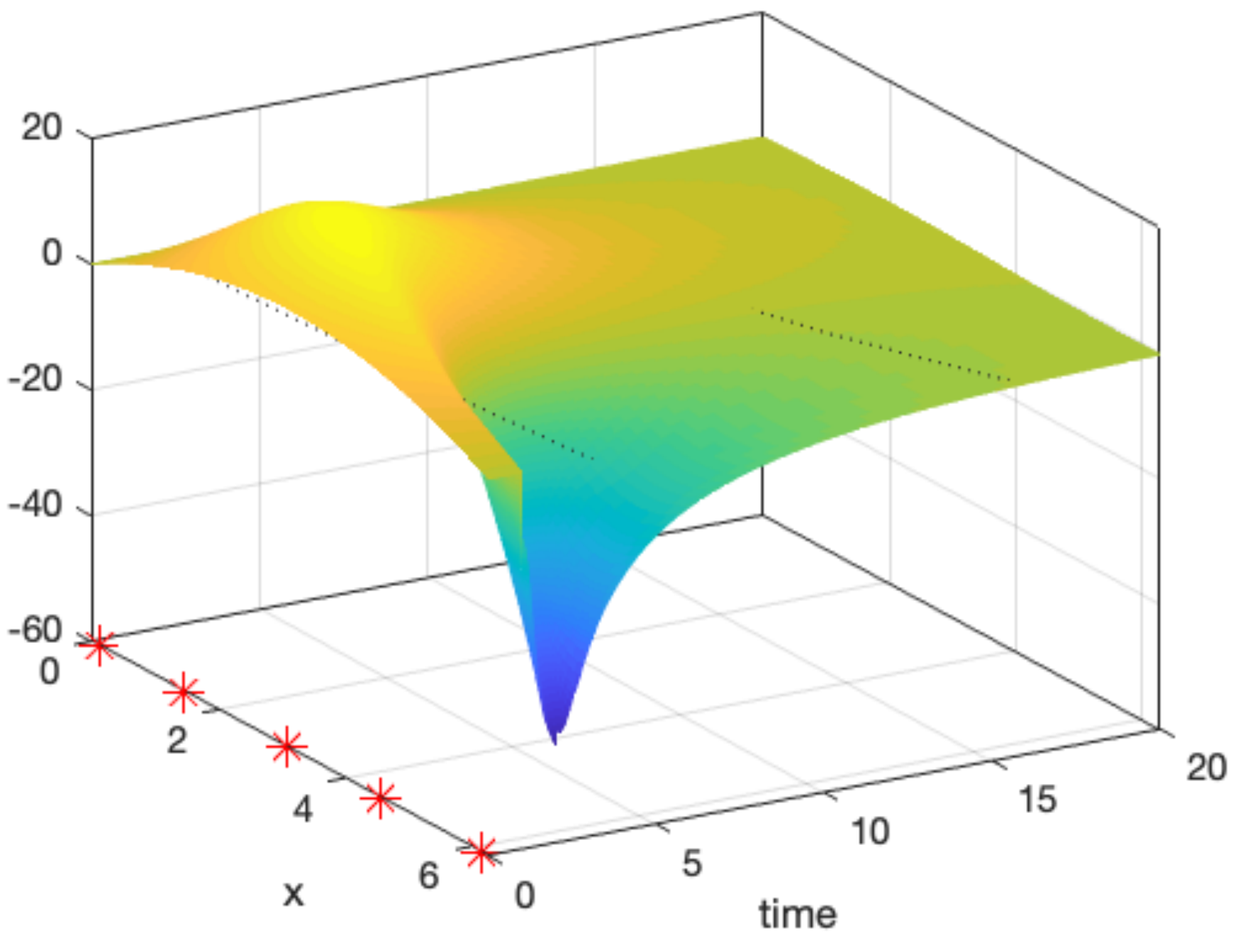}
\caption{Model matching. Simulation of optimized reduced-order system (left) with $K_0$ based on (\ref{H2}) and
 infinite-dimensional system (right) with $H_\infty$-optimal controller $K^*$ based on (\ref{opt1}).  The intermediate result on the left
leads to the final result on the right. The stars '*' indicate sensor positions.  \label{figResp1}}
\end{figure}

\subsection{Mixed sensitivity approach} 
\label{mixed}
A classical approach to improve system performance is via  minimization of the sensitivity function $S:= (I+G K)^{-1}$.  
Using appropriate weighting filters $W_e(s)$, one can achieve better transient  and steady-state responses for all references signals of finite energy. 
Introducing a penalization of the control effort as before leads to a mixed-sensitivity design problem: 
\begin{eqnarray}
\label{program_mixed}
\begin{array}{ll}
\mbox{minimize} & \left\| \left[\begin{array}{c} W_e(s) (I+G(s) K(s))^{-1} \\ W_u(s) K(s) (I+G(s) K(s))^{-1}  \end{array}\right] \right\|_\infty \\
\mbox{subject to} & \mbox{$K$ stabilizes $G$ internally} \\
&K \in \mathscr K_2
\end{array}
\end{eqnarray}
which is a particular case of (\ref{program}).

Actuation takes place at the edge $\xi=L$, and its effect propagates with a unit delay $e^{-s}$ along the spatial dimension. 
The settling time is essentially determined by the slow dynamics that correspond to the states farthest away from the actuation point. 
On the other end, states closer to $\xi=L$ settle faster, but in turn are affected  by much more turbulent transients, as for instance seen in Fig. \ref{figResp20}.  
This suggests shaping response surfaces 
using weightings which
take the distance to $\xi=L$ into account. This leads us to
$W_e = {\rm diag}(c_1,c_2,c_3,c_4,c_5)$ with individually adapted  $c_i$, where for simplicity static weightings are sought. 
Note that $c_1$ corresponds to the edge $\xi=0$, while $c_5$ is associated with $\xi=L$. As before, penalization of the control effort  uses
a high-pass filter $W_u(s) = (s/10)/(1+s/1e3)$. 

\begin{figure}[h!]
\centering
%\vspace*{-.5cm}
\includegraphics[height=5.1cm]{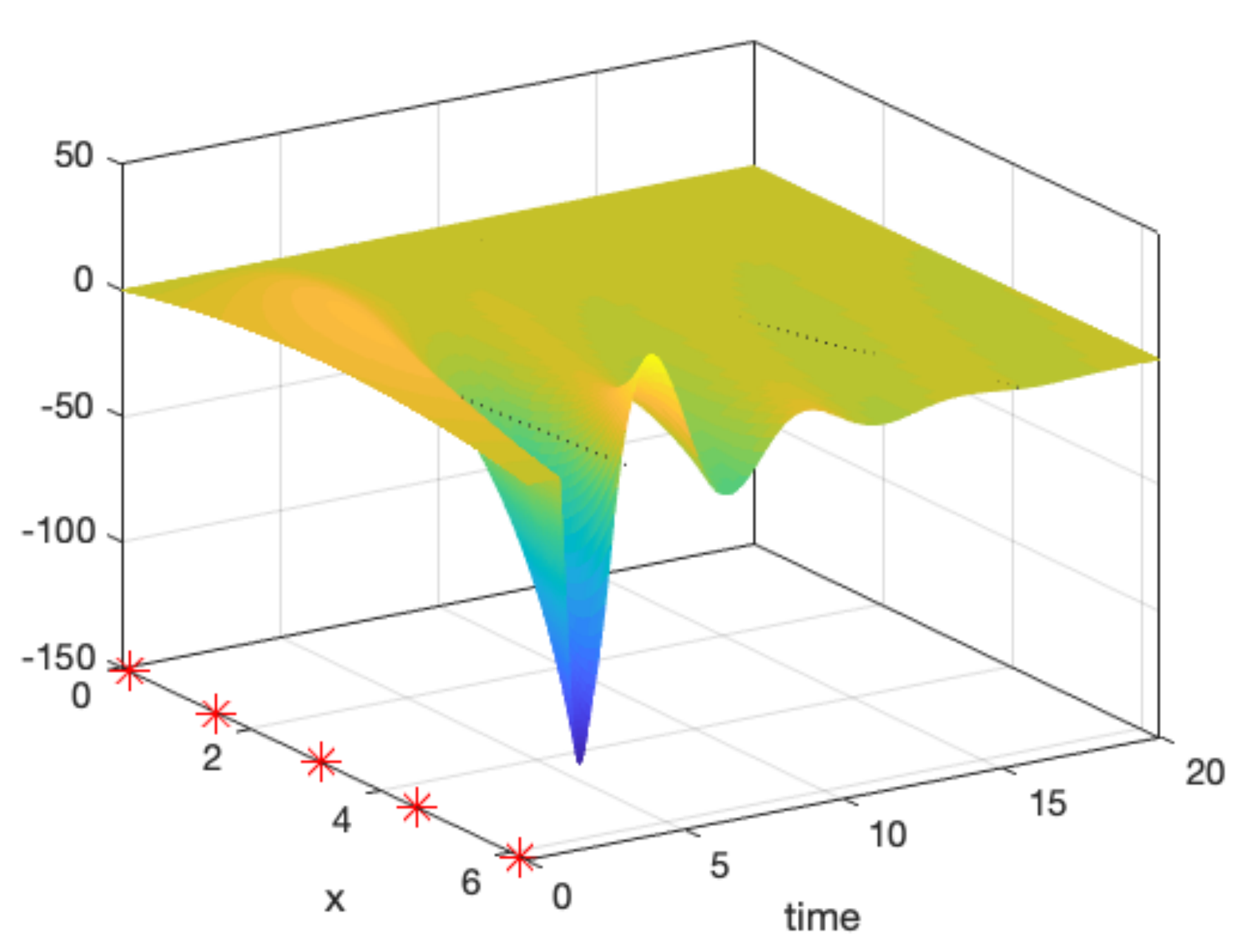}
\includegraphics[height=5.1cm]{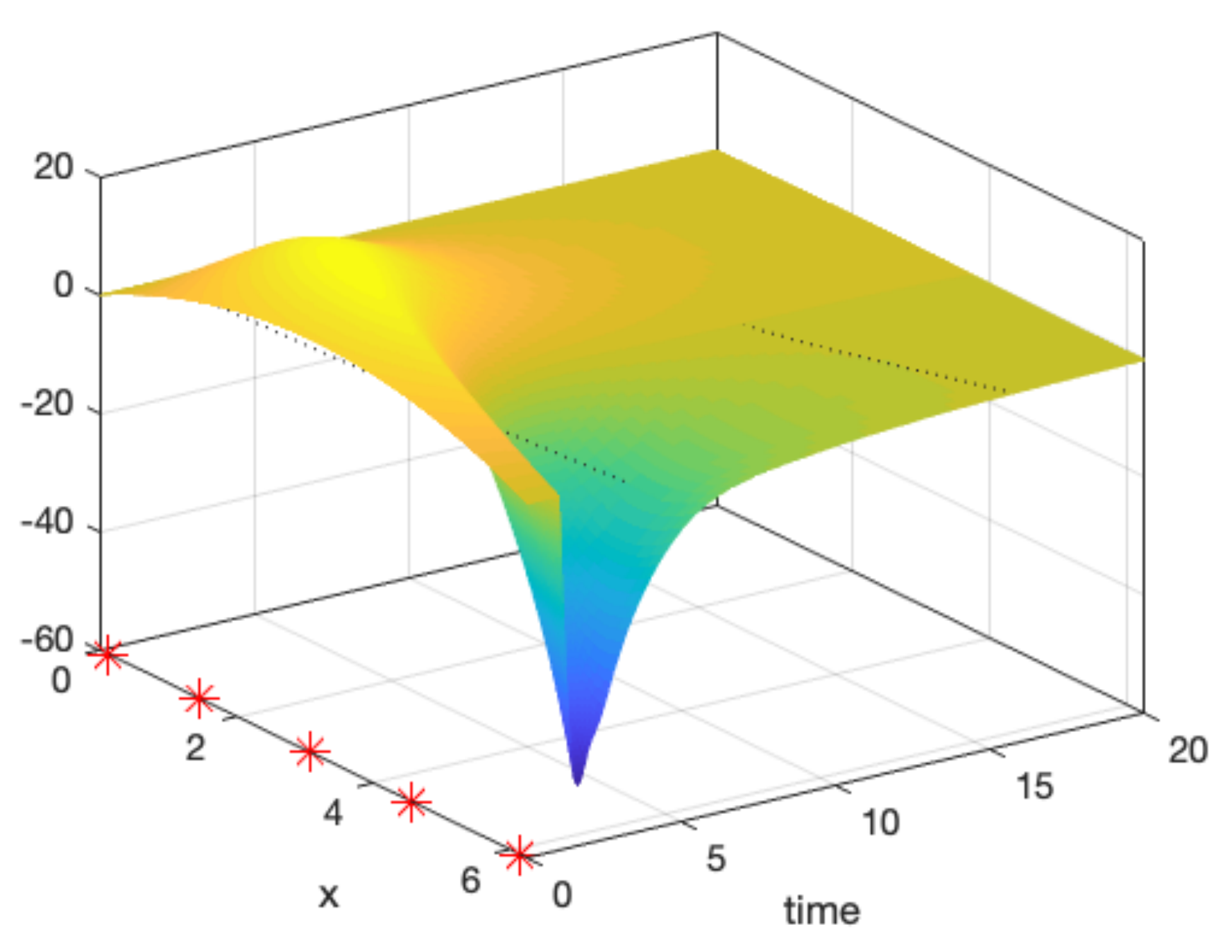}
\caption{Mixed sensitivity. Simulation with $K^*\in \mathscr K_2$ obtained with weight $W_e^{(1)}$ (left).
States near $\xi=L$ show undesirable transient oscillations. Simulation with $K^{**}$ obtained with
$W_e^{(2)}$  (right) proves satisfactory.\label{figResp20}}
\end{figure}

Fig. \ref{figResp20} (left) shows the simulation of $G$ in closed loop with a first optimal controller $K^*\in \mathscr K_2$ given in (\ref{opt2})
obtained via nonsmooth optimization (\ref{program_mixed}) started at $K_0$ from (\ref{K0}) and using $W_e^{(1)} = {\rm diag}(3,0,0,0,0)$. 
States close to the edge $\xi=0$ have excellent  settling times,  but transient 
wobbles manifest themselves at  the opposite end $\xi=L=2\pi$ (see Fig. \ref{figResp20} left). 
\begin{align}
\label{opt2}
K_{11}^*&=\frac{0.0002403 s^2 + 0.3159 s + 2.629}{s^2 + 2.291 s + 19.85} \notag \\
 K_{12}^*&=\frac{0.0125 s^2 + 7.134 s + 37.54}{s^2 + 2.291 s + 19.85} \notag \\
 K_{13}^*&=\frac{-0.02098 s^2 + 6.46 s + 73.02}{s^2 + 2.291 s + 19.85}\\
 K_{14}^*&=\frac{-0.01589 s^2 + 6.447 s + 49.82}{s^2 + 2.291 s + 19.85} \notag \\
 K_{15}^*&=\frac{0.007613 s^2 + 1.283 s + 11.02}{
s^2 + 2.291 s + 19.85}\notag
\end{align}
Increasing the cost at $\xi=L$ via $W_e^{(2)} = {\rm diag}(1,0,0,0,0.2)$, and starting optimization at $K^*$, now
leads to the optimal controller $K^{**}$ given in (\ref{opt3}), which
removes this undesirable effect.  
Simulation of the loop with $K^{**}$ including the control signal  is shown in Fig. \ref{figResp20} (right), the Nyquist plot in \ref{figResp20} (right). 
%The control signal is shown in Fig. \ref{figControlSignal}. 
The  optimal controller is obtained as  
\begin{align}
\label{opt3}
K_{11}^{**}&=\frac{0.00336 s^2 + 0.4678 s + 2.196}{s^2 + 3.731 s + 21.2} \notag\\
K_{12}^{**}&=\frac{-0.002542 s^2 + 6.097 s + 21.47}{s^2 + 3.731 s + 21.2}\notag \\
 K_{13}^{**}&=\frac{0.08966 s^2 + 3.947 s + 33.65}{s^2 + 3.731 s + 21.2}\\
 K_{14}^{**}&=\frac{ -0.01911 s^2 + 5.889 s + 27.07}{s^2 + 3.731 s + 21.2}\notag \\
 K_{15}^{**}&=\frac{-0.006395 s^2 + 0.7398 s + 5.143}{s^2 + 3.731 s + 21.2} \notag
 \end{align}
Finally, note that it is possible to obtain even faster responses by accepting more aggressive control signals and therefore obtaining a 
more academic than practical solution. By proposition \ref{prop1}
all controllers obtained are exponentially stabilizing.

%
%\begin{figure}[h!]
%\centering
%%\vspace*{-.5cm}
%\includegraphics[height=5cm]{figMethod2KFinal.png}$\;$
%\caption{Mixed sensitivity. Simulation of  infinite-dimensional system $G$ in closed loop with controller $K^{**}$ associated to $W_e^{(2)}$. %\newline 
% \label{figResp2}}
%\end{figure}

%\begin{figure}[h!]
%\centering
%%\vspace*{-.5cm}
%%\includegraphics[height=5cm]{figMethod2FinalNyquist.png}$\;$
%\caption{Mixed sensitivity. Nyquist plot  of  controller $K^{**}$  optimized with $W_e^{(2)}$  and the infinite-dimensional system $G$ \label{figNyq2}}
%\end{figure}

\begin{figure}[h!]
\centering
%\vspace*{-.5cm}
%\includegraphics[height=5cm]{controlSignal.png}$\;$
\vspace*{-.1cm}
\includegraphics[height=5.1cm]{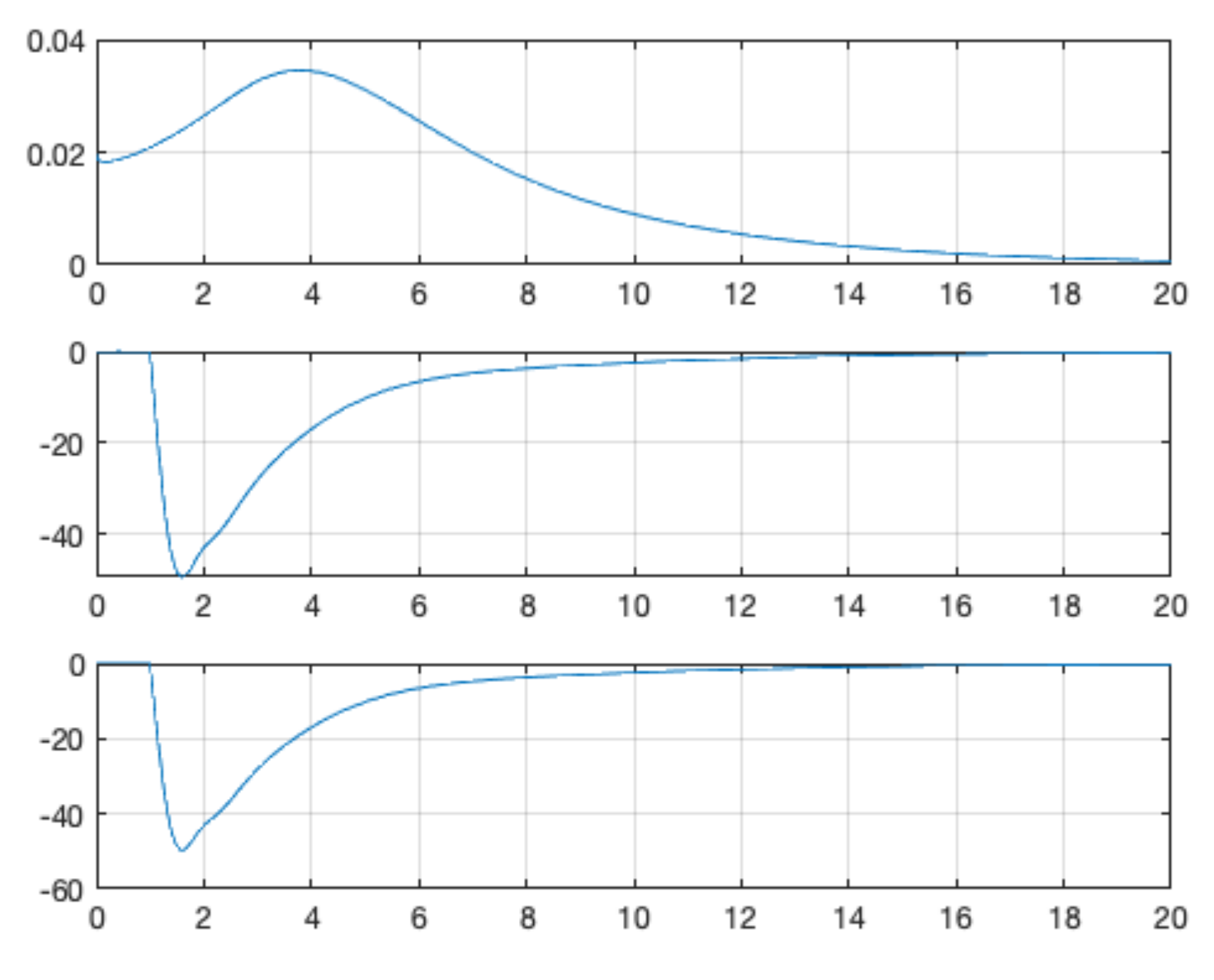}
\includegraphics[height=5.3cm]{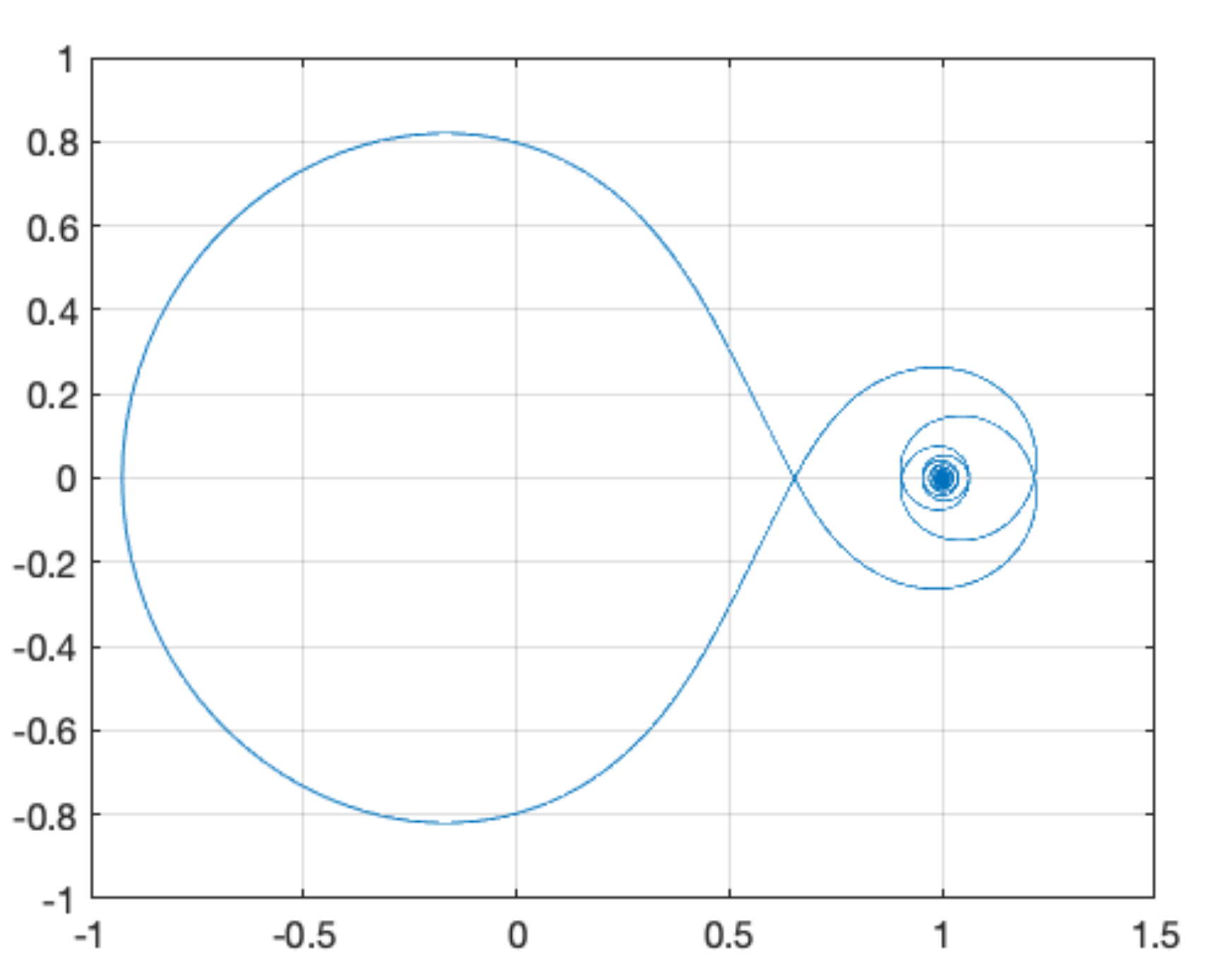}
\caption{Mixed sensitivity. Left. Simulations with controller $K^{**}$: top $x(L,t)$, middle $x(0,t)$ and bottom control signal $u(t)$. 
Right. Nyquist curve of final $K^{**}$.\label{figControlSignal}}
\end{figure}

Altogether this study shows that by way of program (\ref{program}) it is possible to conveniently control
the reaction-diffusion equation (\ref{diff}) with a single input with delay and 5 distributed measurements by synthesizing a finite-dimensional low-order controller  
such that the result matches the result obtained in \cite{prieur_trelat:17} using full state feedback.

\section{Control  of an anti-stable wave equation}
\label{sect_wave}
In this second study, we discuss the following boundary control system
\begin{align}
\label{wave}
x_{tt}(\xi,t) &= x_{\xi\xi}(\xi,t), \; t \geq 0, \xi \in [0,1] \notag \\
x_\xi(0,t) &= -q x_t(0,t) \\
x_\xi(1,t) &= u(t),\notag
\end{align}
where $q > 0$,  $q \not=1$. The state of the system is $x(\cdot,t),x_t(\cdot,t)$, the control 
applied at the boundary $\xi=1$ is $u(t)$, and we assume that the measured
outputs are 
$$y_1(t) = x(0,t), y_2(t)=x(1,t) \mbox{  and } y_3(t) = x_t(1,t).$$ 
The system has been discussed previously
in \cite{Fridman:2015}, and  \cite{smyshlyaev_krstic:09,bresch_krstic:14},  where potential applications are mentioned. Its
well-posedness can be seen from the functional analytic set-up in  \cite{smyshlyaev_krstic:09,bresch_krstic:14}, and from the general approach
to well-posedness of 1D hyperbolic systems in \cite{Zwart:2010}.
%In these works controllers are computed using the backstepping approach.

The transfer function of (\ref{wave}) is obtained from the elliptic boundary value problems
\begin{align}
\label{wave}
s^2x(\xi,s) &= x_{\xi\xi}(\xi,s), \; s\in \mathbb C, \xi \in [0,1] \notag \\
x_\xi(0,s) &= -q sx(0,s) \\
x_\xi(1,s) &= u(s),\notag
\end{align}
which in this particular situation can be solved analytically:
\[
G(\xi,s) =\frac{x(\xi,s)}{u(s)} = \frac{1}{s} \cdot \frac{(1-q)e^{s\xi}+(1+q)e^{-s\xi}}{(1-q)e^s - (1+q)e^{-s}}.
\]
From this general formula the transfer function of (\ref{wave})
is %obtained as 
$G(s) = \left[ G(0,s);G(1,s);sG(1,s)\right] =: \left[ G_1(s);G_2(s);G_3(s)\right]$.
The main challenge in the hyperbolic system (\ref{wave}) is  that along with the unstable pole at $s=0$ it exhibits an infinite number of unstable 
poles on a line Re$(s) = \sigma > 0$. This means that the Nyquist test is not directly applicable.

\subsection{Preliminary stabilization}
Following our scheme in algorithm \ref{algo1},  the first step is to provide a preliminary
stabilizing controller $K_0 = K(\x^0)$ of a simple pre-defined structure.  
We have to stabilize the system
$$G(s) = \begin{bmatrix} \frac{2e^{-s}/(1-q)}{s(1-Qe^{-2s})}  \vspace{.2cm} \\  \frac{1+Qe^{-2s}}{s(1-Qe^{-2s})}  
\vspace{.2cm} \\ \frac{1+Qe^{-2s}}{1-Qe^{-2s}}
\end{bmatrix}
=\begin{bmatrix}G_1(s) \vspace{.15cm} \\ G_2(s)  \vspace{.15cm} \\G_3(s)\end{bmatrix}$$
where $Q = (1+q)/(1-q)$.

A first question is 
whether $G$ can be stabilized by a finite-dimensional controller. 
Ignoring the input $y_2$, which for stabilization is not required, we choose the structure
$K(\x) = \left[ n_1(s)/d(s),0,n_3(s)/d(s)\right]$, with $\x$ gathering the unknown coefficients
of the polynomials $n_i(s), d(s)$ with deg$(n_i)\leq {\rm deg}(d)$. Stability of the closed loop
$T(s)=G(s)/(1+G_1(s)K_1(s)+G_3(s)K_3(s)$ leads to testing whether the quasi-polynomial
\[
(1-q)s(d(s)+n_3(s)) + (1-q)sQ e^{-2s} (n_3(s)-d(s)) + 2n_1(s)e^{-s}
\]
arising in the denominator of $T(s)$ is stable, i.e., has its roots in $\mathbb C^-$. 
While there exist general methods to check stability of quasi-polynomials, cf. \cite{Pont42},  an
 {\em ad hoc}  solution is here to choose $n_3=d$,
whence the quasi-polynomial simplifies to
\[
s d(s) + c(s)e^{-s},
\]
where $c(s) = n_1(s)/(1-q)$ and deg$(n_1)={\rm deg}(c)\leq {\rm deg}(d)$.
If we choose $d(s) = s + x_1$ and $c(s) = x_2s+x_3$, then stability of the loop is
equivalent to stability of the quasi-polynomial
\[
P(s) = A(s) + B(s)e^{-s}, A(s) = s^2 + x_1 s, B(s) = x_2s + x_3,
\]
which is covered by the discussion in \cite{MM:06}.  In their terminology we have $a_0=0$, $a_1=x_1$,
$b_2 = 0$, $b_1=1$, $b_0=1$. We are then necessarily in the case $m=1$, $\mu_0=0$ of \cite{MM:06},
so the quasi-polynomial  $P(s)$ can only be stable if $x_1 > -1$.
Moreover the family $P_h(s) = A(s) + B(s)e^{-hs}$ is stable for all $0 \leq h < h_{\sigma,0}$, where
$h_{\sigma,0}>0$ is determined as follows. Let $\omega_\sigma$ be the positive real solution of
\[
4\omega_\sigma^3 - 2 \omega_\sigma (1-x_1^2) = \textstyle\frac{1}{2} \sqrt{5-2x_1^2+x_1^4}
\] 
and let $h_{\sigma,0}$ be the smallest positive  solution $h$ of
\[
\omega_\sigma h = \arg\left( - \frac{B(j \omega_\sigma)}{A(j\omega_\sigma)}   \right) + 2k\pi, k\in \mathbb N,
\]
where $\arg(\cdot)\in [0,2\pi)$. If we let $x_1 = 1 > -1$, then
$4\omega_\sigma^3 = 1$, $\omega_\sigma = 4^{-3}$,
and we get
\[
h_{\sigma,0} = 4^3 \arg\left( - \frac{x_2j 4^{-3} + x_3}{4^{-6} + 4^{-3}}  \right),
\]
and since our delay is $h=1$, this must now be solved for $x_2,x_3$  so that $h_{\sigma,0} > 1$. For instance
$x_2=-1$ and $x_3 =-4^{-3}$ gives argument $\pi/4$ in the formula, so that
$h_{\sigma,0} = 16\pi > 1$.  The leads to the finite-dimensional stabilizing controller $K_0$ for $G$:
\begin{equation}
\label{K0}
K_0 = \left[  \frac{(1-q) (s+4^{-3})}{s+1}    \; \; 0\;\;1 \right]. 
\end{equation}

%\begin{remark}
%%For $q=1$ the transfer function $G(\xi,s)$ is still well-posed, but the sensor $x_t(1,t)$ is no longer a reasonable choice. 
%For $q=0$ the unstable poles are on $j\mathbb R$, and the system cannot be stabilized by a finite-dimensional controller.  For $q < 0$ the 
%open loop is stable.
%\end{remark}

A second way to seek preliminary stabilization of (\ref{wave}) is to stick to the form
$K = \left[ \frac{n(s)}{d(s)} \; \;0 \;\; 1\right]$, but allow $n(s),d(s)$ to be quasi-polynomials, trying to simplify
the denominator quasi-polynomial $P$ as much as possible. A very straightforward way is to let
$d(s) =a(s) + e^{-s} b(s)$ with $a(s), b(s)$ polynomials, then the denominator quasi-polynomial simplifies to
$$\frac{e^{-s}}{2(1-q)} \left( s a(s) +  (s b(s)+c(s)) e^{-s}  \right),$$ where $c(s) = n_1(s)/(1-q)$. If we now let $c(s) = -sb(s)$,
then $K$ will be stabilizing in the $H_\infty$-sense as soon as $a(s)$ is stable, because the factor $s$ cancels with the factor $s$ in the numerator. If we choose $a(s) = s+c_0$,
$b(s) = -c_0$ for some constant $c_0 > 0$, then we obtain the controller
\begin{equation}
\label{bresch}
K = \left[ \frac{c_0(1-q)s}{s+c_0(1-e^{-s})}  \;\;0\;\;  1 \right],
\end{equation}
which in  \cite{bresch_krstic:14} was obtained using the back-stepping technique. Since only input and output delays along with real-rational terms arise,
such controllers are implementable, so we are still in line with our general purpose of computing
practically useful controllers.

%{\color{blue}
%\subsection{Exponential stability of the loop}
%The question whether the controllers (\ref{K0}) or (\ref{bresch}) are exponentially stabilizing
%is settled by Theorem \ref{theorem2}.  All we have to do is show that (\ref{wave}) is exponentially
%stabilizable and detectable. While this could be done separately, we use
%%\cite{gouais}, where the authors stabilize one unstable wave equation by another unstable wave equation exponentially (with regard to the given
%state-space realizations).
%Since our system $G$ corresponds to \cite[(2)]{gouais}, and is stabilized by their system (1), we get 
%simultaneously stabilizability and detectability.  This means that (\ref{K0}), (\ref{bresch}),
%as well as all other LTI-controllers synthesized later on in this section, stabilize not only in the $H_\infty$-sense, but
%even exponentially with the appropriate state-space representations.
%}

\subsection{Performance optimization}
\label{perf_opt}
Let us now discuss 
a more systematic way which not only leads to preliminary  stabilizing $K_0$,
but also allows performance optimization. In order to compare with  \cite{bresch_krstic:14}, we 
optimize again against the effect of non-zero initial values, using the output $y\in \mathbb R^3$, and
aiming as before at a convenient implementable controller structure.

We start by putting the system $G$ in feedback with the controller $K_0=[ 0\; 0\; 1]$, which leads to
$\widehat{G} = G/(1+G_3)$, where
$$G(s) = \begin{bmatrix} \frac{2e^{-s}/(1-q)}{s(1-Qe^{-2s})}   \vspace{.2cm}   \\    \frac{1+Qe^{-2s}}{s(1-Qe^{-2s})}   \vspace{.2cm} \\ \frac{1+Qe^{-2s}}{1-Qe^{-2s}}  \end{bmatrix},
\qquad 
\widehat{G}(s)= \begin{bmatrix}  \frac{1}{s(1-q)}  \vspace{.2cm}    \\ \frac{1+Q}{2s} 
 \vspace{.2cm}\\ \frac{1}{2}\end{bmatrix}
+ \begin{bmatrix}  -\frac{1-e^{-s}}{s(1-q)} \vspace{.2cm}    \\ -\frac{Q(1-e^{-2s})}{2s} \\ \vspace{.2cm}\frac{Q}{2} e^{-2s}\end{bmatrix}.
$$
Re-write this as $\widehat{G} = \widetilde{G} + \Phi$, where $\widetilde{G}$ is now real-rational and still unstable, while $\Phi$ gathers the infinite dimensional part,
but is stable. Then we use that  stability of the closed loop $(\widetilde{G}+\Phi,K)$ is equivalent to
stability of  the loop  $(\widetilde{G}, {\tt feedback}(K,\Phi))$, as explained in Fig. \ref{swap1}:

\begin{figure}[ht!]
\includegraphics[scale=1.2]{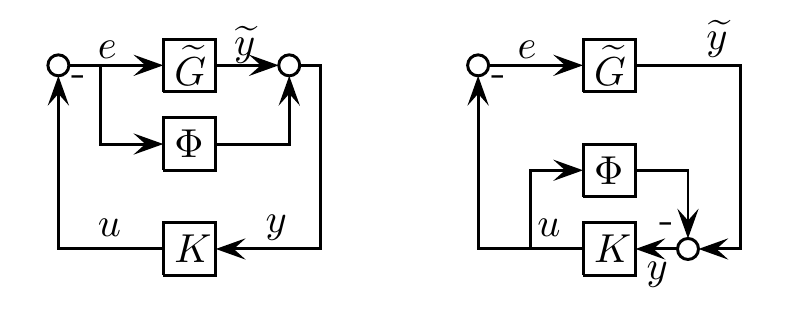}
\caption{Stability of the closed-loop $(\widetilde{G}+\Phi,K)$ is equivalent to stability of the closed-loop $(\widetilde{G},{\tt feedback}(K,\Phi))$.
%See also \cite{Moelja_Meinsma:03}.
 \label{swap1}}
\end{figure}

%\vspace*{.4cm}
\noindent
Then we construct a finite-dimensional structured
controller $\widetilde{K}=\widetilde{K}(\x)$ which stabilizes $\widetilde{G}$. The controller $K$ is then recovered from $\widetilde{K}$ through the
equation $\widetilde{K} = {\tt feedback}(K,\Phi)$, which when inverted gives
$K = {\tt feedback}(\widetilde{K},-\Phi)$. The overall controller is then
$K^* = K_0 + K$. 

Construction of $\widetilde{K}$ uses {\tt systune} where we use pole placement
via {\tt TuningGoal.Poles} imposing that closed-loop poles have a minimum decay of 0.9, minimum damping of 0.9, and a maximum frequency of 4.0. The controller structure is chosen as
static, so that $\x \in \mathbb R^3$. A simulation with $K^*$
is shown in Fig. \ref{Krstic_controller} (right) and some acceleration over the simulation for  backstepping controller
(left) from the same nonzero initial value is observed.

\begin{figure}[ht!]
\centering
\includegraphics[height=5.2cm]{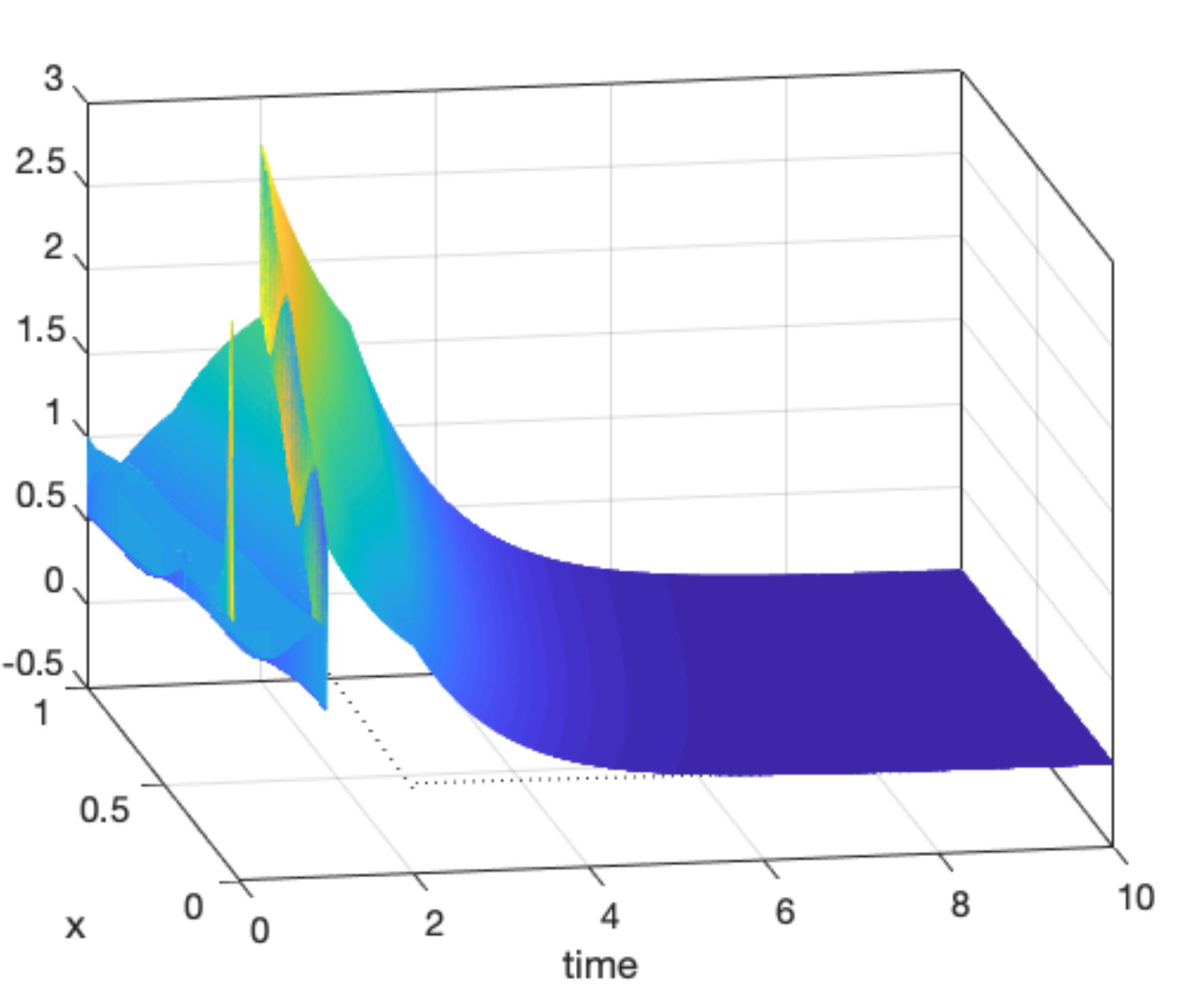}
\includegraphics[height=5.2cm]{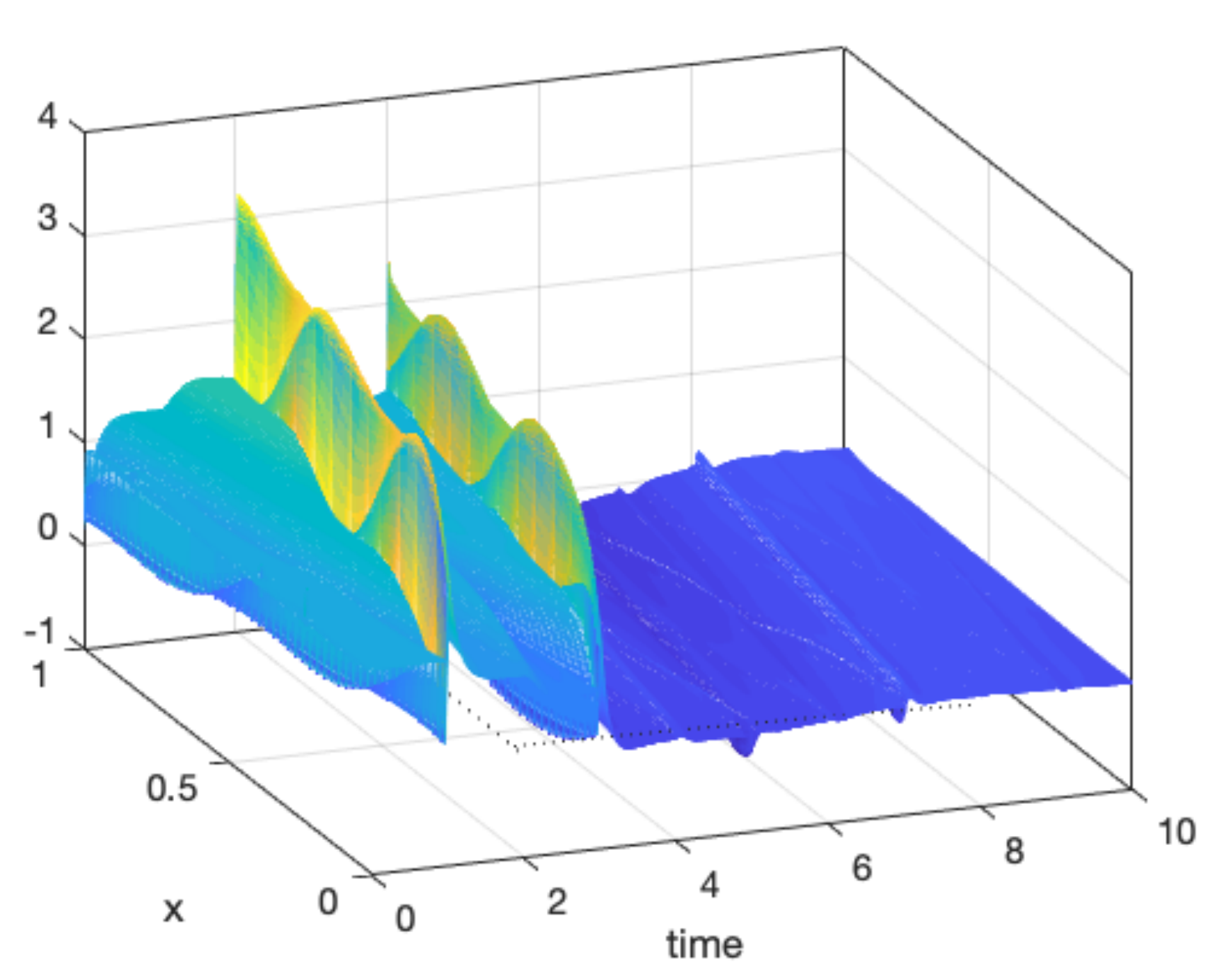}
\caption{Wave equation. Simulations with nonzero initial condition for $K$ obtained by backstepping control (left)
and  $K^*=K_0+K$ obtained by optimizing ${\tt feedback}(\widetilde{G},\widetilde{K})$ via {\tt systune} (right).
%$x_0(\xi) = 0.1 \sin(\xi)$. 
Both controllers are infinite-dimensional, but implementable.
\label{Krstic_controller}}
\end{figure}

%\begin{figure}[ht!]
%\centering
%\includegraphics[height=5cm]{result2.png}
%\caption{     Wave equation. Controller $K^*=K_0+K$ is obtained by optimization of  ${\tt feedback}(\widetilde{G} ,\widetilde{K})$   via
%{\tt systune}.  Simulation with non-zero initial condition $x_0(\xi)=0.1\sin(\xi)$.         \label{optimized_controller}}
%\end{figure}

%\begin{figure}[ht!]
%\centering
%\includegraphics[height=5cm]{draft.png}
%\caption{  ???? Can be removed. Dommage, car les cascades d'ondes sont jolies  \label{poleplace_controller}}
%\end{figure}

We recall that the controllers obtained in this section stabilize the system
exponentially, as follows from the scheme on the right of Fig. \ref{swap1}, where $\widetilde{G}$, $\widetilde{K}$
are finite-dimensional and $\Phi$ is stable. This shows that $G$ is exponentially stabilizable and detectable, so that every $H_\infty$-stabilizing
controller is also exponentially stabilizing.  In particular, this applies retro-actively also to the controllers
(\ref{K0}) and (\ref{bresch}).

In more detail, we have to relate the trajectories of the system on the right of Fig. \ref{swap1} to the trajectories
on the left.  Here we can follow \cite{Moelja_Meinsma:03}, using their formulas (13), (14), to the extent that 
exponential decay of trajectories on the right of Fig. \ref{swap1} as assured by the finite-dimensional stabilization
achieved via {\tt systune} leads to exponential decay of trajectories on the left of Fig. \ref{swap1}, which uses of
course exponential stability of the infinite-dimensional part $\Phi$. While in \cite{Moelja_Meinsma:03} work with strictly proper
$G,K$, it suffices for our present argument to suppose that all loops are well-posed.  This is for instance guaranteed for a proper $K$,
$\widetilde{K}$,
since $\Phi$ is strictly proper.

%\section{Application}\label{sect-Application} 
%\begin{table}[!h]
%\caption{default}
%\begin{center}
%\begin{tabular}{|c|c|}\hline
%Measurement locations & $H_2$ norm \\ \hline
%$[1,\,50]$ &  $8.93$  \\ \hline
%$[1,\,25,\,50]$ &  $5.12$  \\ \hline
%$[1,\,12,\,37,\,25]$ &  $5.10$  \\ \hline
%$[1,\,12,\,25,\,37,\,50]$ &  $4.81$  \\ \hline
%$1,\,2,\,3,\ldots,\,50$ &  $4.76$  \\ \hline
%
%
%
%\end{tabular}
%\end{center}
%\label{default}
%\end{table}%

\subsection{Performance with finite-dimensional control}
In this section, we show that the anti-stable wave equation (\ref{wave})
may be regulated satisfactorily with a simple 3rd-order finite-dimensional
controller. We initialize our procedure with the controller
$K_0$ in (\ref{K0}) obtained via the quasipolynomial test. Then
we write the desired structure
$K(\x)$ as $K(\x) = K_0 + K_1(\x)$, where $K_1(\x)=\left[{n_1}/{d}\; {n_2}/{d} \;{n_3}/{d}  \right]$
and deg$(n_i) \leq 2$, deg$(d)=2$, which requires $11$ variables. This is a subclass of 
the class of 3rd-order controllers.

According to  section \ref{G0}, we consider the pre-stabilized system
$G_0=G(I+K_0 G)^{-1}$, build the closed loop
${\tt feedback}(G_0,K_1(\x))$, and 
find an initial $\x_0\in \mathbb R^{11}$ such that $K_1(\x_0)$ is stable and 
$\|K_1(\x_0)\|_\infty < 1/\|G_0\|_\infty$, so that by the Small Gain Theorem the loop
$T(G_0,K_1(\x_0))$ is stable. This is achieved e.g. by
$K_1(\x_0)=n_0/d_0 [1\; 1 \; 1]$ with
$n_0(s)=0.3218 s + 0.0643$, $d_0(s)=s^2+100.1 s+ 10$.  Since $G_0$ has
zero unstable poles, and since $K_1(\x)$ is not allowed unstable poles,
the Nyquist curve $1+G_0K_1(\x)$ turns now zero times around the origin,
and this is maintained during optimization.

We now use the mixed sensitivity approach of section \ref{mixed} again, but under the form
\begin{eqnarray}
\label{finite_dim_K_opt}
\begin{array}{ll} 
\mbox{minimize} & \left \|  \begin{bmatrix} W_e(I+G_0 K_1(\x))^{-1} \\W_uK_1(\x)(I+G_0K_1(\x))^{-1} 
%\\ W_y G_0K_1(\x)(I+G_0K_1(\x))^{-1}
 \end{bmatrix}    
 \right\|_\infty \\
\mbox{subject to} & K_1(\x) \mbox{ stabilizes } G_0 \\
&\x \in \mathbb R^{11}
\end{array}
\end{eqnarray}
where we still have to choose the filters. The $3 \times 3$ filter $W_e$ is 
chosen  diagonal 
$$W_e(s)={\rm diag} \left[ \frac{0.01s+0.5002}{s+0.01429} \; \frac{.99s+0.0007147}{s+0.07941}\; 0.01    \right],$$ where the first entry is a typical low-pass, which corresponds to
the output $e_1$. The transfer $G_2$  is non-minimum phase with unstable zeros
at the positions $-\log(1/2)/2+jk\pi$, $k\in \mathbb Z$, which makes the choice
of the second filter diagonal
element challenging. The above choice turns out to be a good solution, as it forces 
tight control in the high-frequency range  beyond
the first unstable zero at   $-\log(1/2)/2=0.346$. As third weight we choose a simple static
gain $0.01$. The static filter $W_u = 0.01$
serves to avoid unrealistic control  signals.

\begin{figure}[ht!]
\centering
\includegraphics[height=5cm]{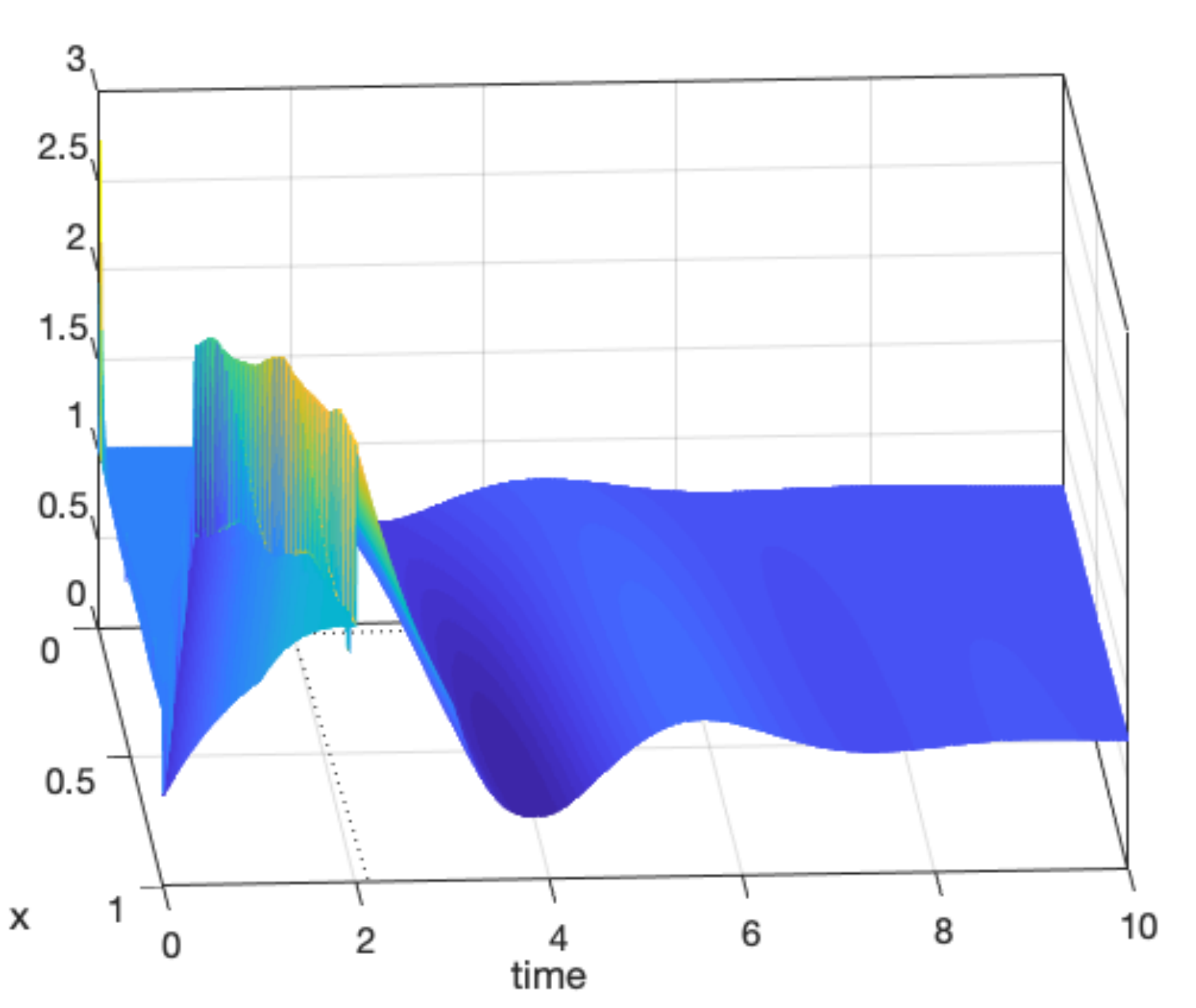}$\quad$
\includegraphics[height=4.4cm,width=5.5cm]{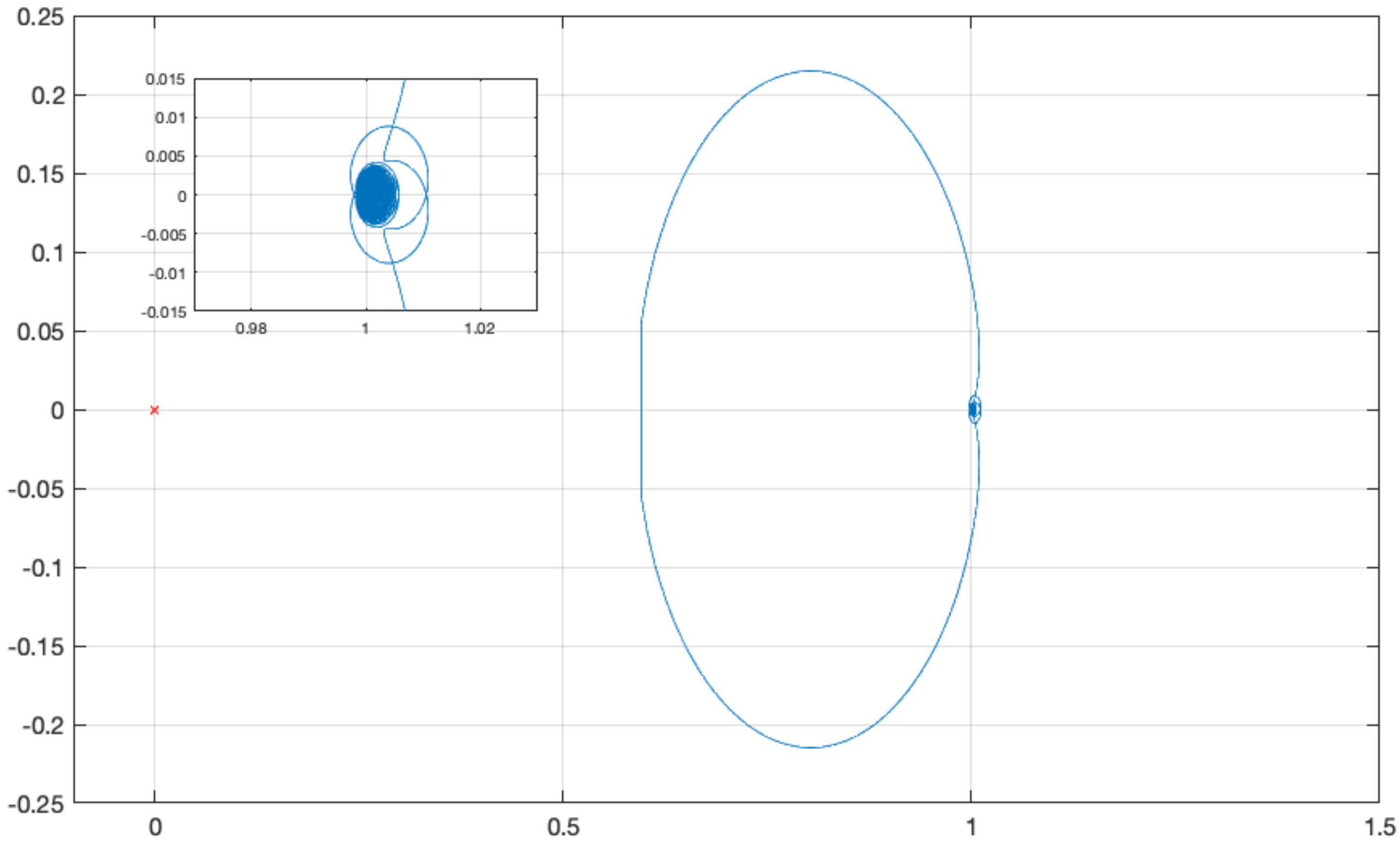}

\includegraphics[height=5cm]{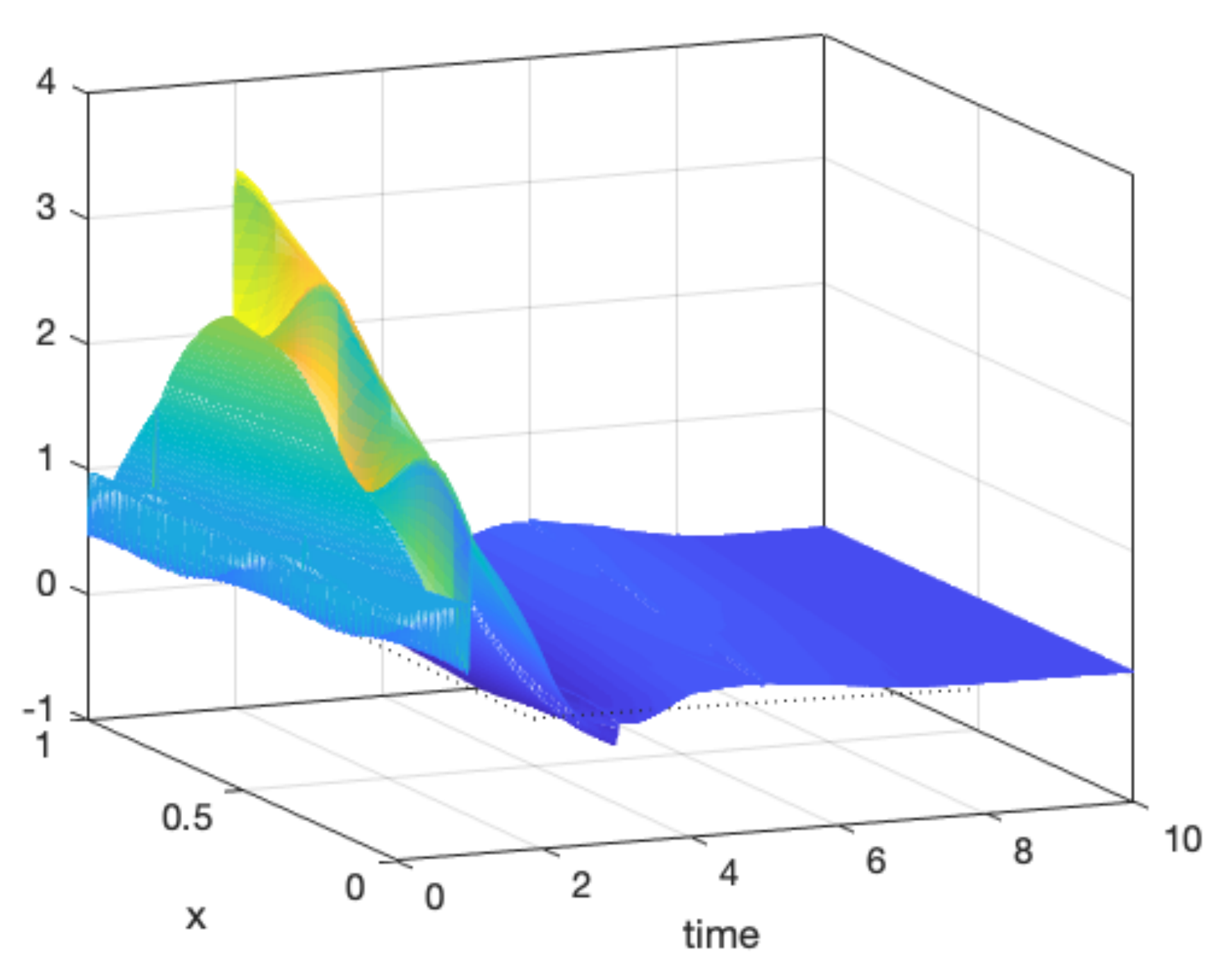}$\quad$
\includegraphics[height=4.6cm,width=5.5cm]{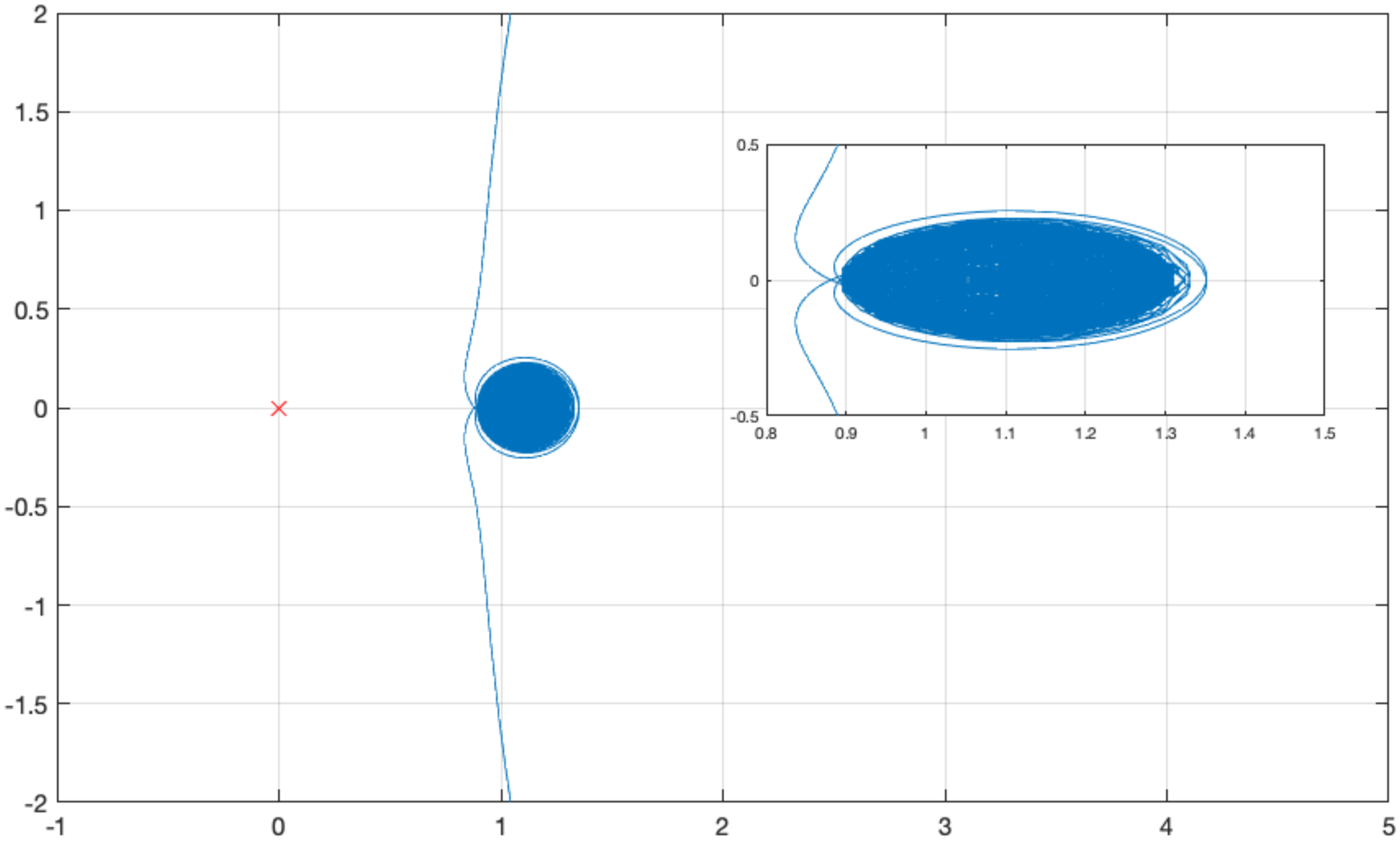}
\caption{Wave equation. Finite-dimensional controllers obtained by mixed sensitivity  in (\ref{finite_dim_K_opt}). Top left initial
$K_0$. Bottom left optimal $K^\sharp=K_0+K_1(\x^\sharp)$.  Top right, Nyquist plot  $1+(K(\x_0)-K_0) G_0$  does not encircle
origin, bottom right Nyquist plot $1+(K(\x^\sharp)-K_0)G_0$. \label{finite_dim_K}}
\end{figure}

The final controller based on (\ref{finite_dim_K_opt}) obtained is
\begin{align}
K_1 &= \frac{-2.992 s^3 - 303.5 s^2 - 104.7 s - 0.488}{s^3 + 102.2 s^2 + 101.7 s + 0.522}   \notag \\
K_2 &= \frac{-0.04494 s^2 - 4.047 s + 0.001097}{s^2 + 101.2 s + 0.522} \\
K_3 &= \frac{1.207 s^2 + 122.7 s + 0.5271}{s^2 + 101.2 s + 0.522}   \notag
\end{align}
The final $H_\infty$-norm in (\ref{finite_dim_K_opt}) was $1.99$,
with  approximately  $1000$ frequencies for both stability and performance. 
In (\ref{finite_dim_K_opt}), we have also constrained the controller to have a
minimum decay rate of $1e-3$ and minimum damping of $0.1$ to keep control on the frequency  inter-sample behavior \cite{AN:18}. 
Furthermore the constraint $|(1+K_1(\x) G_0)^{-1}| \leq 1/0.5$ stands for a disk margin of $0.5$ hence prohibiting  any change in the winding number.

Simulations are shown in Fig. \ref{finite_dim_K}. 
%Note that due to the small gain restriction on $K_1(\x_0)$, simulations of $K_0$ and $K_0 + K_1(\x_0)$ are nearly indistinguishable. 
Top left shows simulation with $K_0+K_1(\x_0)$, bottom left shows the optimized controller
$K_0 + K_1(\x^\sharp)$, achieving
faster convergence  and a much smaller
smaller steady-state error beyond 
$4$ sec. 
Simulations of the slices $\xi = 0$, $\xi = 1$ and the control signal are displayed in Fig. \ref{figSlices} from top to bottom and confirm the previous analysis.

\begin{figure}[ht!]
\centering
\includegraphics[height=6cm]{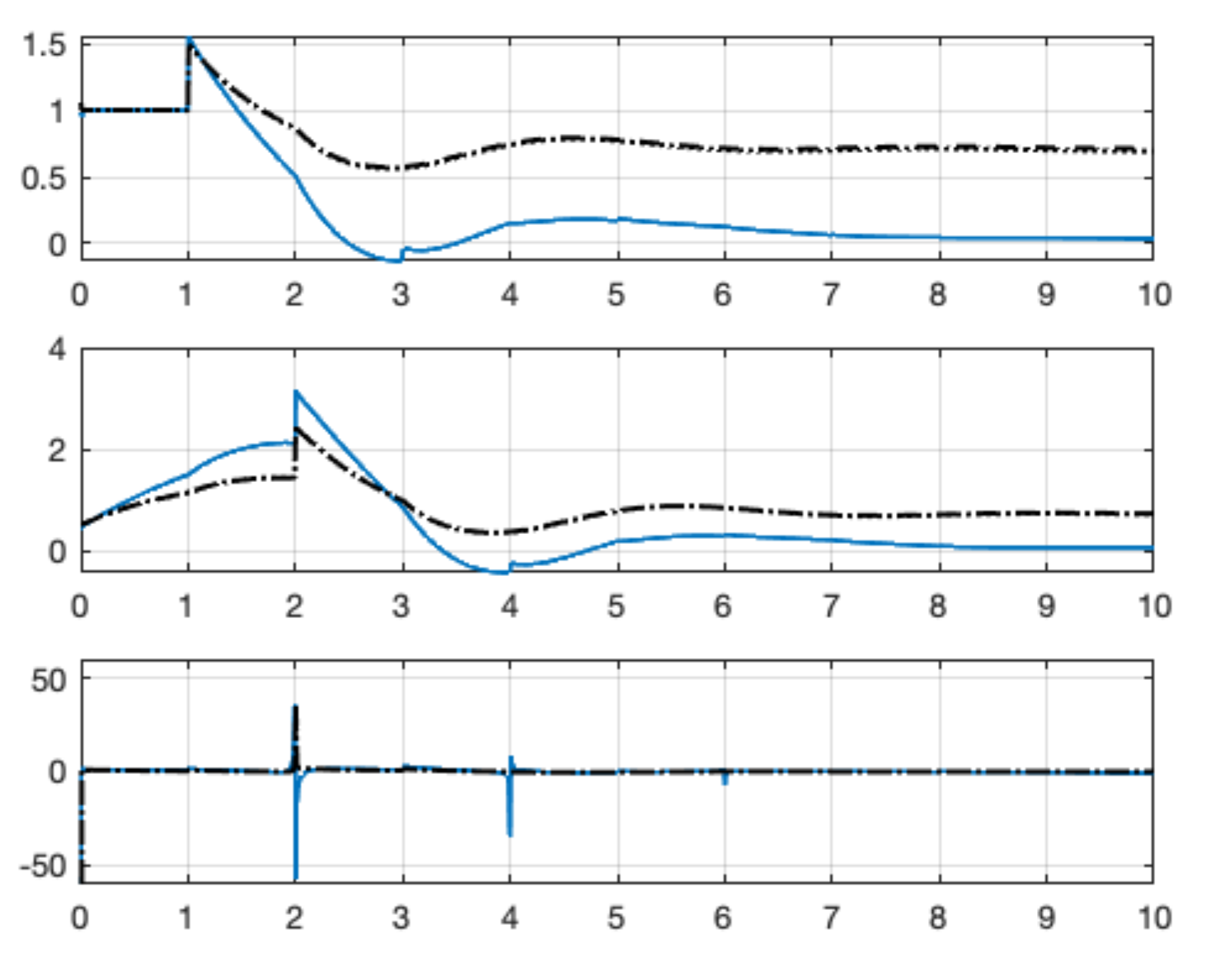}
\caption{Wave equation. Simulations of  slices $\xi = 0$, $\xi = 1$  and control signal from top to bottom.  $K^\sharp$: solid blue, $K_0$: dotted black, $K_0+K_1(\x_0)$: dashed black. 
Simulations of $K_0$ and $K_0 + K_1(\x_0)$ are nearly indistinguishable due to small gain restriction on $K_1(\x_0)$.
\label{figSlices}}
\end{figure}

\subsection{Gain-scheduling control}
Our last study concerns the case where the parameter $q\geq 0$
is uncertain or allowed to vary in time with sufficiently slow variations as discussed in \cite{shamma90}. We assume that
a nominal value $q_0>0$ and an uncertain interval $[\underline{q},\overline{q}]$
with $q_0 \in (\underline{q},\overline{q})$ are given. The authors of \cite{bresch_krstic:14}
schedule their controller (\ref{bresch}) using an adaptive control scheme, where the scheduling
function uses a nonlinear dynamic estimate
$\widehat{q}(t) \in [\underline{q},\overline{q}]$ of the anti-damping parameter.

Based on the approach in section \ref{perf_opt} the following scheduling
scenarios are possible.
(a) Computing a nominal controller $\widetilde{K}$ at $q_0$ as before, and 
scheduling through $\Phi(q)$, which depends explicitly on $q$, so that
$K^{(1)}(q) = K_0+{\tt feedback}(\widetilde{K},-\Phi(q))$. (b) Computing a
$\widetilde{K}(q)$ which depends already on $q$, 
and using $K^{(2)}(q) =K_0+ {\tt feedback}(\widetilde{K}(q),-\Phi(q))$.
(c) Computing a robust controller $K_{\rm rob}$  for the entire interval. 

While (a) is directly based on (\ref{program}) in its finite-dimensional version based on \cite{AN06a,an_disk:05}, see also \cite{AN15},
as available in {\tt systune},  leading to $K^{(1)}(q)$, we show that one can also apply our approach to case (b). We use the reduction of 
section \ref{perf_opt}, see Fig. \ref{swap1}, to work in the finite-dimensional system $(\widetilde{G}(q),\widetilde{K}(q))$, where we now have 
in addition dependency on
$q$, addressed by a parameter-varying  design.

For that we have to decide on
a parametric form of the controller $\widetilde{K}(q)$, which we chose here as
\[
\widetilde{K}(q,\x) = \widetilde{K}(q_0) + (q-q_0) \widetilde{K}_1(\x) + (q-q_0)^2 \widetilde{K}_2(\x),
\]
and 
where we adopted the simple static form
 $\widetilde{K}_1(\x) =[\x_1\; \x_2\; \x_3], \widetilde{K}_2=[\x_4\;\x_5\; \x_6]$, featuring a total of 6 tunable parameters.
The nominal $\widetilde{K}(q_0)$ is obtained via the
synthesis technique in section \ref{perf_opt}. For $q_0 = 3$ this leads to 
$\widetilde{K}(q_0)=[ -1.049 \; -1.049 \; -0.05402]$, obtained
via {\tt systune} as in section \ref{perf_opt}.

With the parametric form $\widetilde{K}(q,\x)$ fixed,  we now use again the feedback system $(\widetilde{G}(q),\widetilde{K}(q))$ in
Fig. \ref{swap1} and design a parametric robust controller using the method of 
\cite{AN:2015}, which is implemented in the {\tt systune} package and used by default if an
uncertain closed-loop is entered. The tuning goals are chosen as constraints on closed-loo poles
including minimum decay of 0.7, minimum damping of 0.9, with maximum frequency $2$.
The controller obtained is (with $q_0=3$)
\[
\widetilde{K}(q,\x^*) = \widetilde{K}(q_0) + (q-q_0) \widetilde{K}_1(\x^*) + (q-q_0)^2 \widetilde{K}_2(\x^*),
\]
with numerical values 
$\widetilde{K}_1 =   [-0.1102, -0.1102, -0.1053]$,  
$\widetilde{K}_2=     [0.03901, 0.03901,  0.02855]$, 
and we retrieve the final parameter varying controller for the system $G(q)$ as
\[
K^{(2)}(q) = K_0 + {\tt feedback}(\widetilde{K}(q,\x^*),-\Phi(q)).
\]
The methods are compared in simulation in Figs. 
\ref{nominal_robust},
\ref{method1},
\ref{method2}.
Comparison of the simulations in Figs. \ref{nominal_robust}, \ref{method1}, and \ref{method2} indicates that the last controller $K_3(q)$ achieves the best
performance for frozen-in-time values $q\in [2,4]$.

\begin{figure}[ht!]
\includegraphics[width=0.32\textwidth]{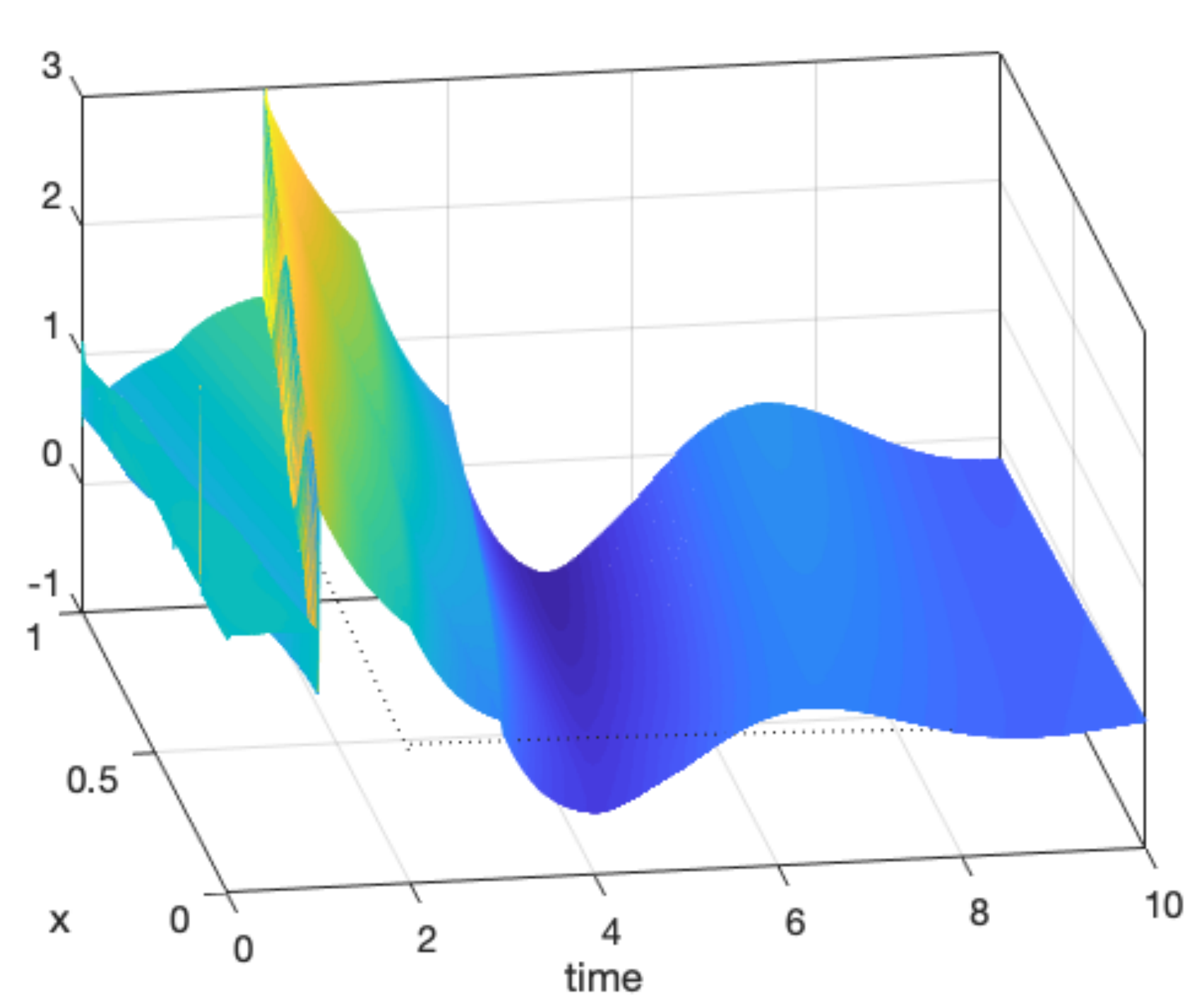}
\includegraphics[width=0.32\textwidth]{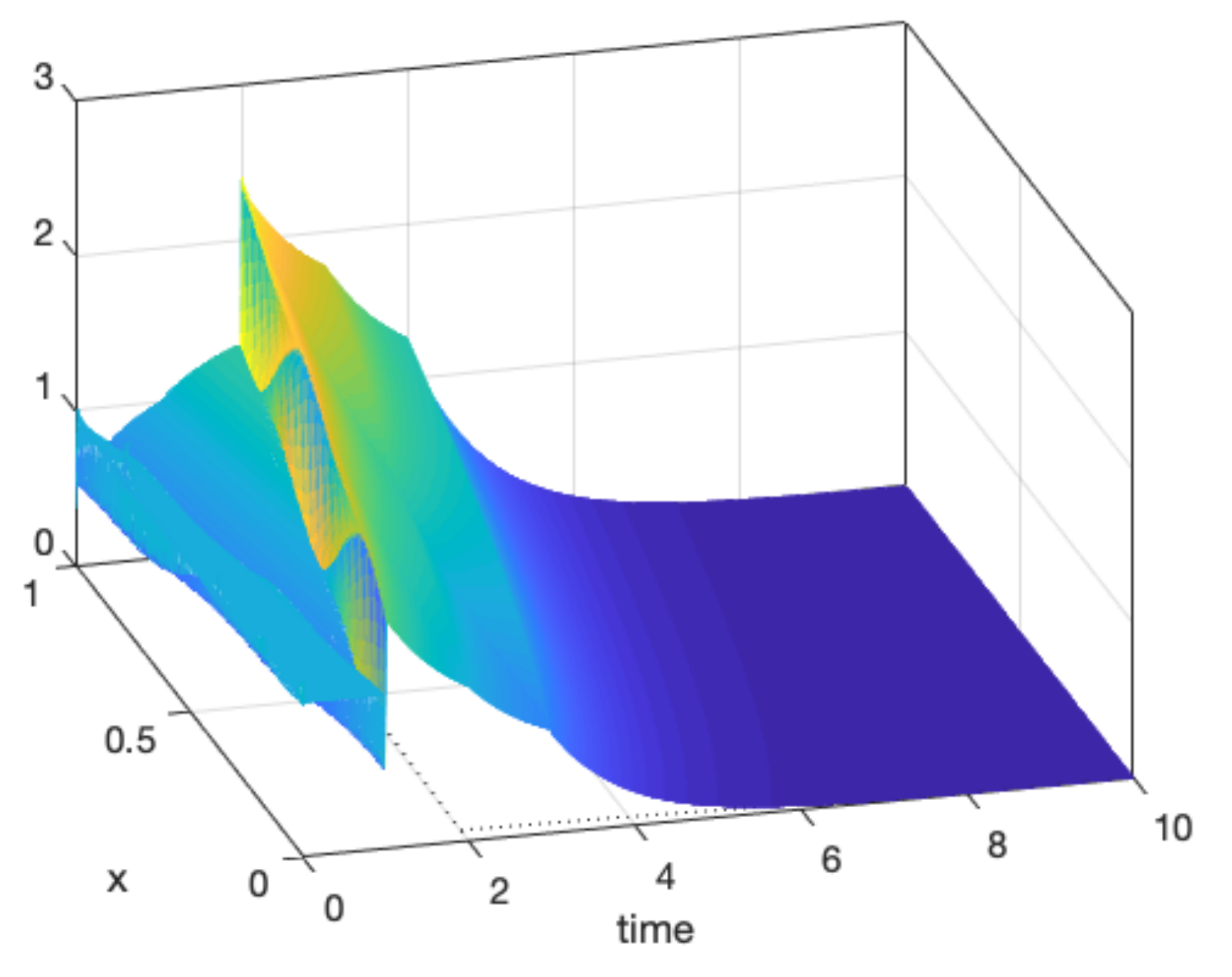}
\includegraphics[width=0.32\textwidth]{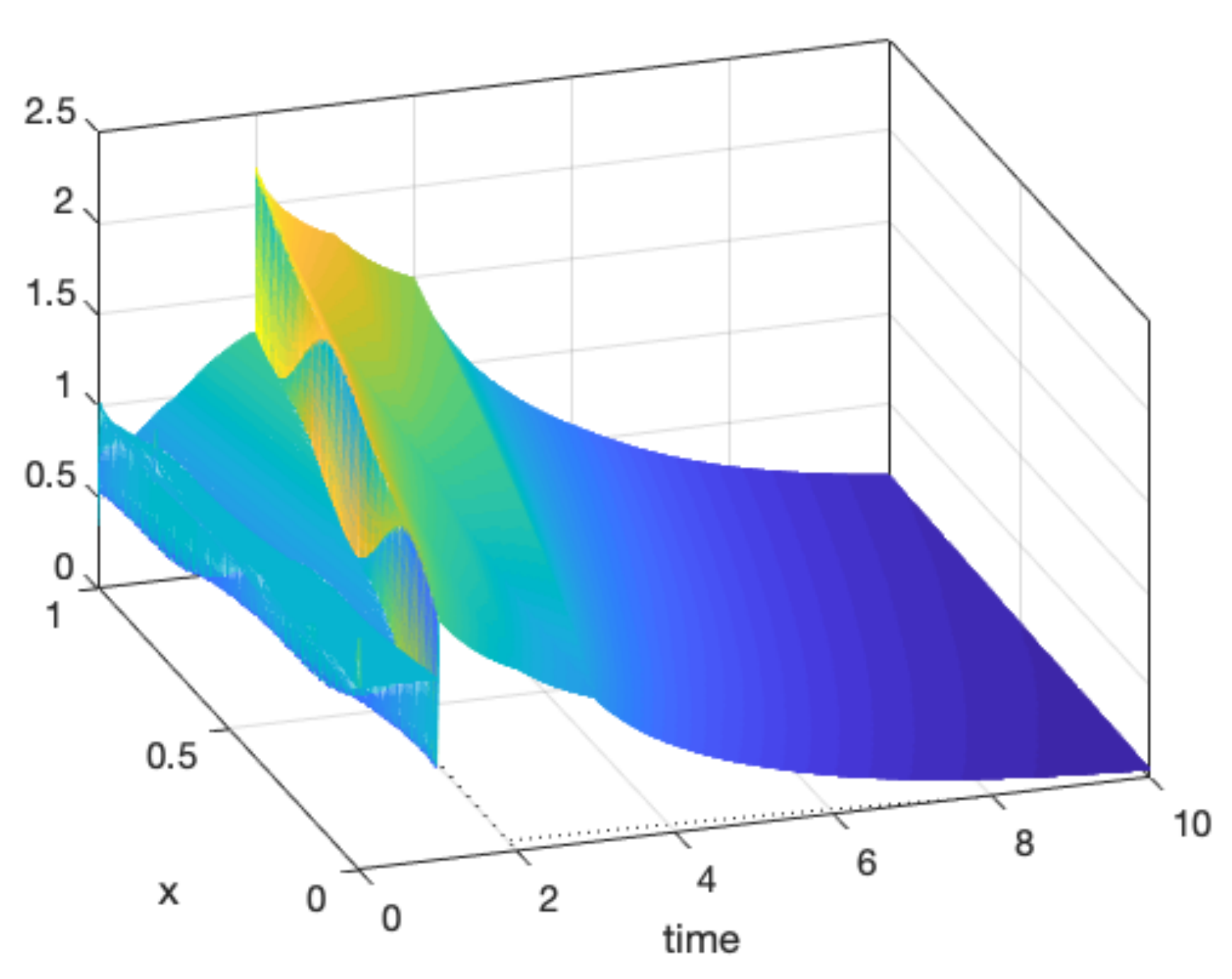}
\caption{Synthesis at nominal $q_0=3$. Simulations of nominal $K=K_0+{\tt feedback}(\widetilde{K},\Phi(3))$  for $q=2,3,4$. Nominal controller is robustly
stable over $[\underline{q},\overline{q}]$.
\label{nominal_robust}}
\end{figure}

\begin{figure}[ht!]
\includegraphics[width=0.32\textwidth]{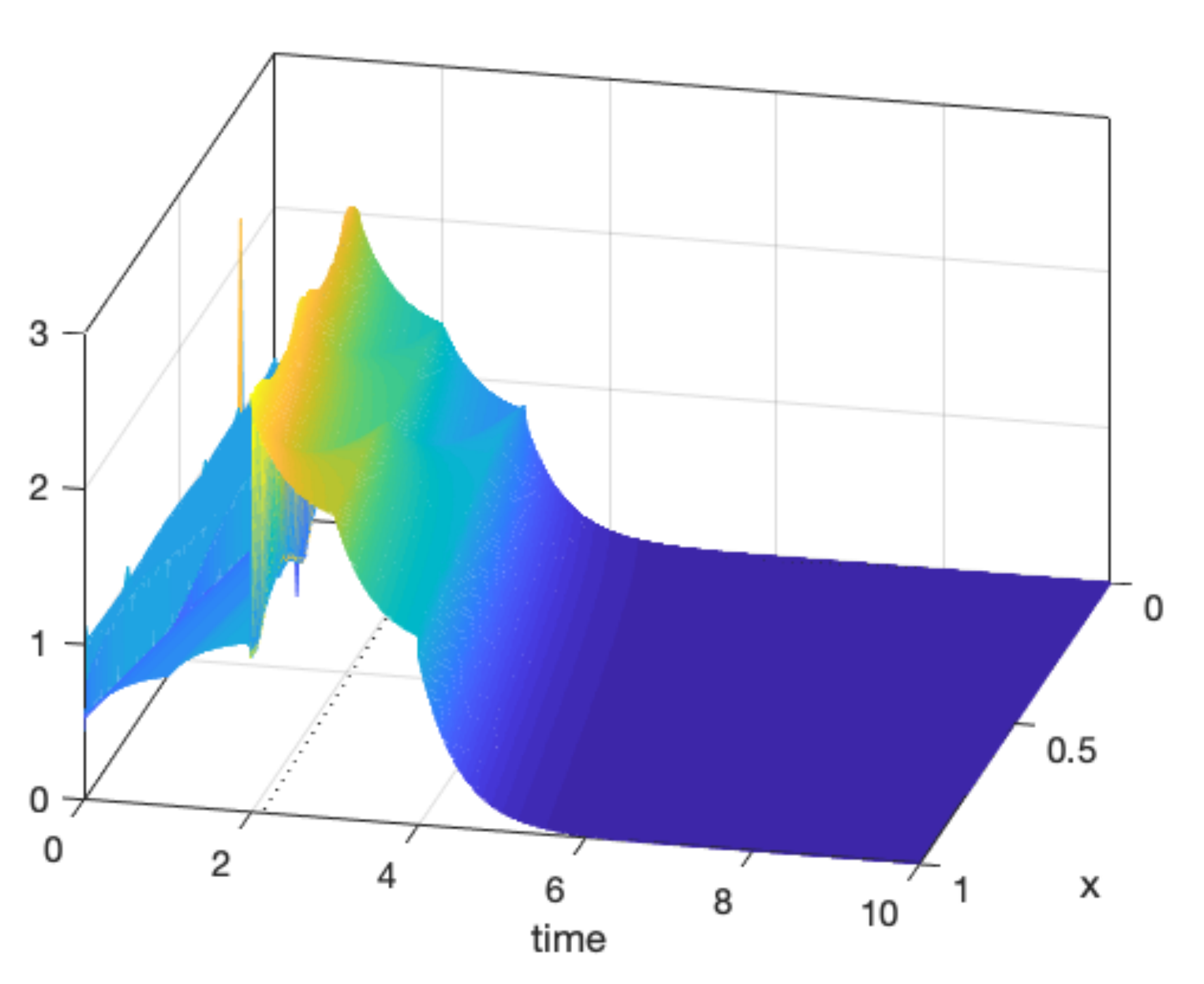}
 \includegraphics[width=0.32\textwidth]{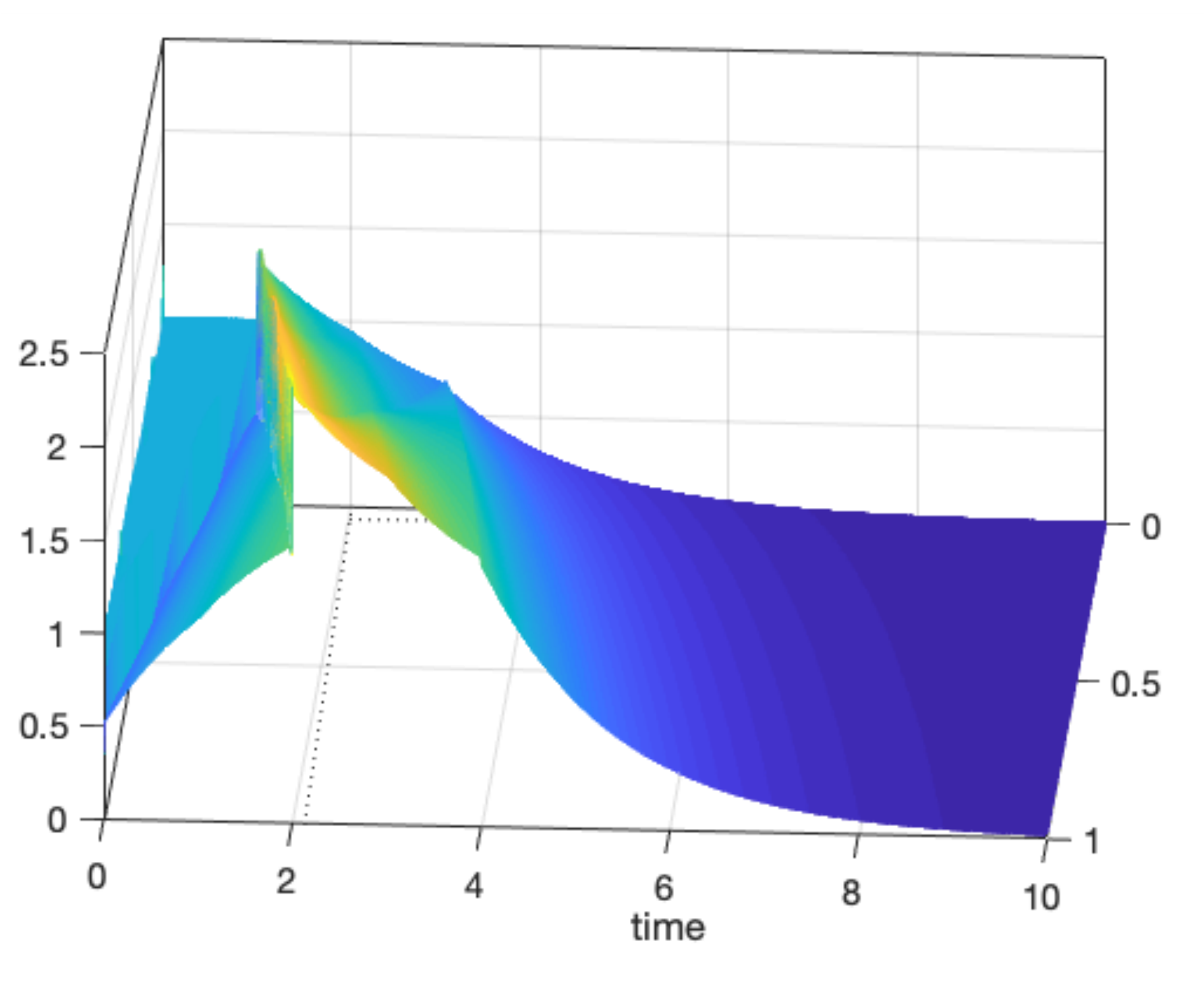}
\caption{Method 1.  $\widetilde{K}$ obtained for nominal $q=3$, but scheduled
$K(q) =K_0+ {\tt feedback}(\widetilde{K},\Phi(q))$. Simulations for $q=2$ left, $q=3$ middle,
$q=4$ right \label{method1}}
\end{figure}

\begin{figure}[ht!]
\includegraphics[width=0.32\textwidth]{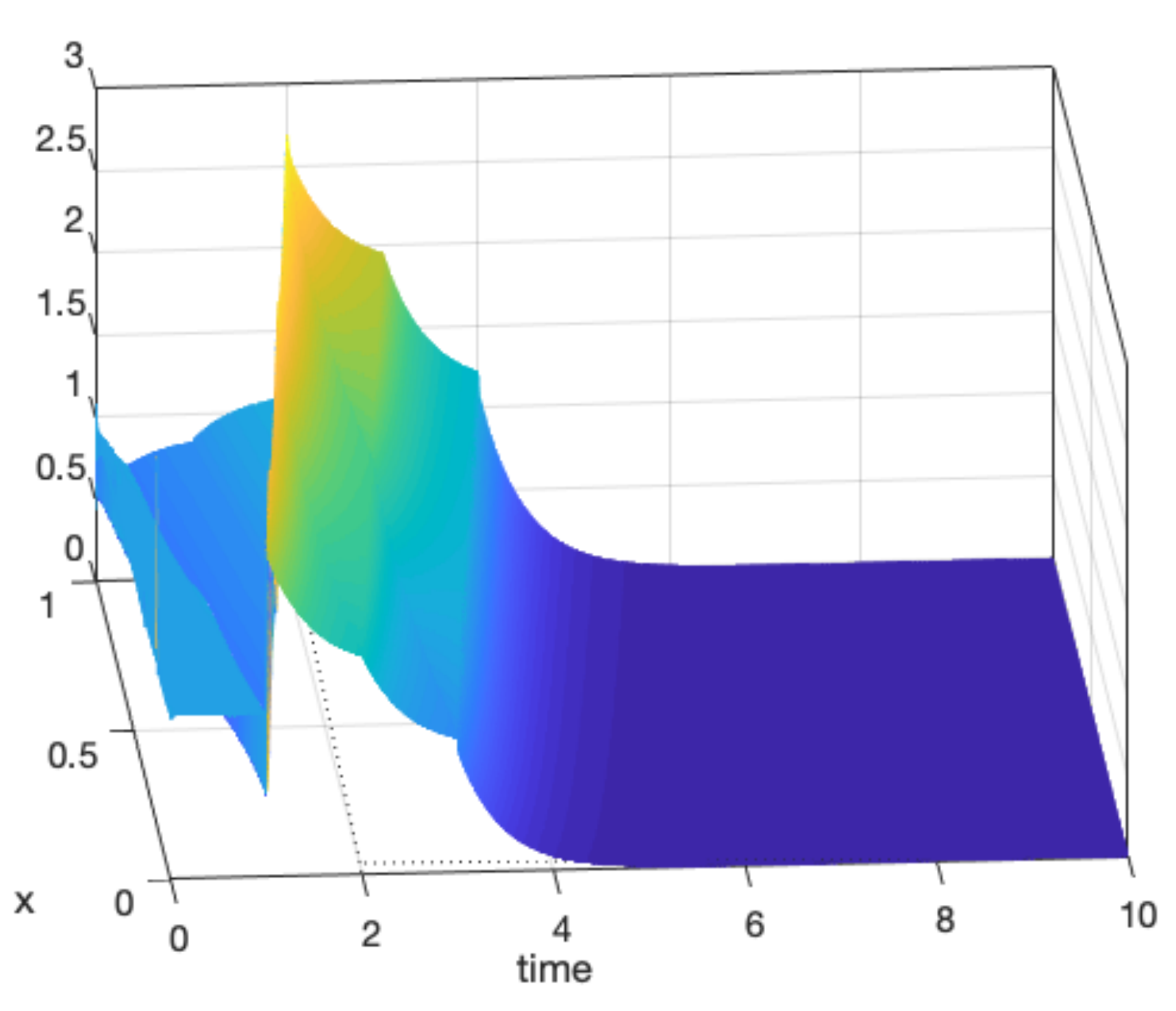}
\includegraphics[width=0.32\textwidth]{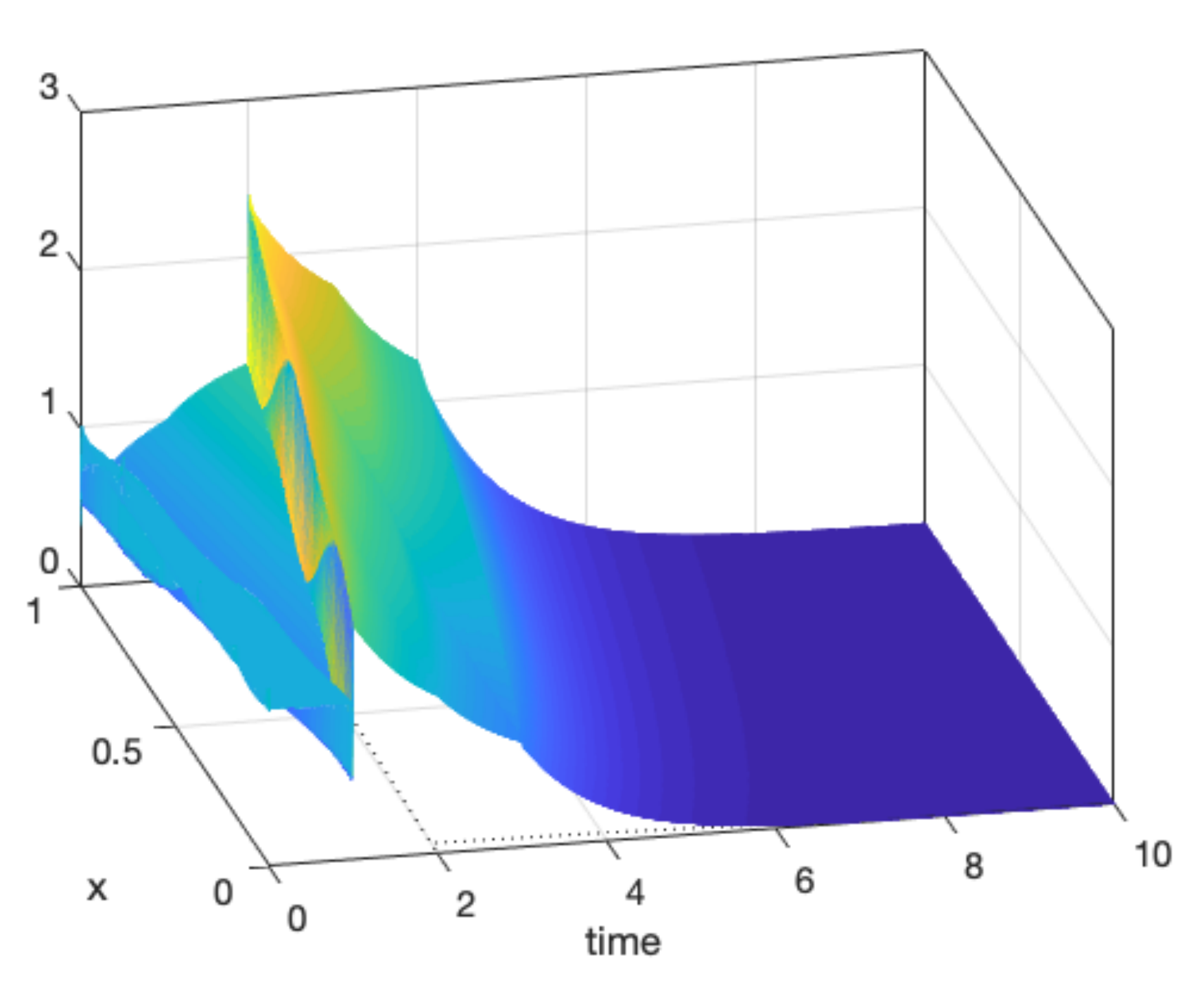}
\includegraphics[width=0.32\textwidth]{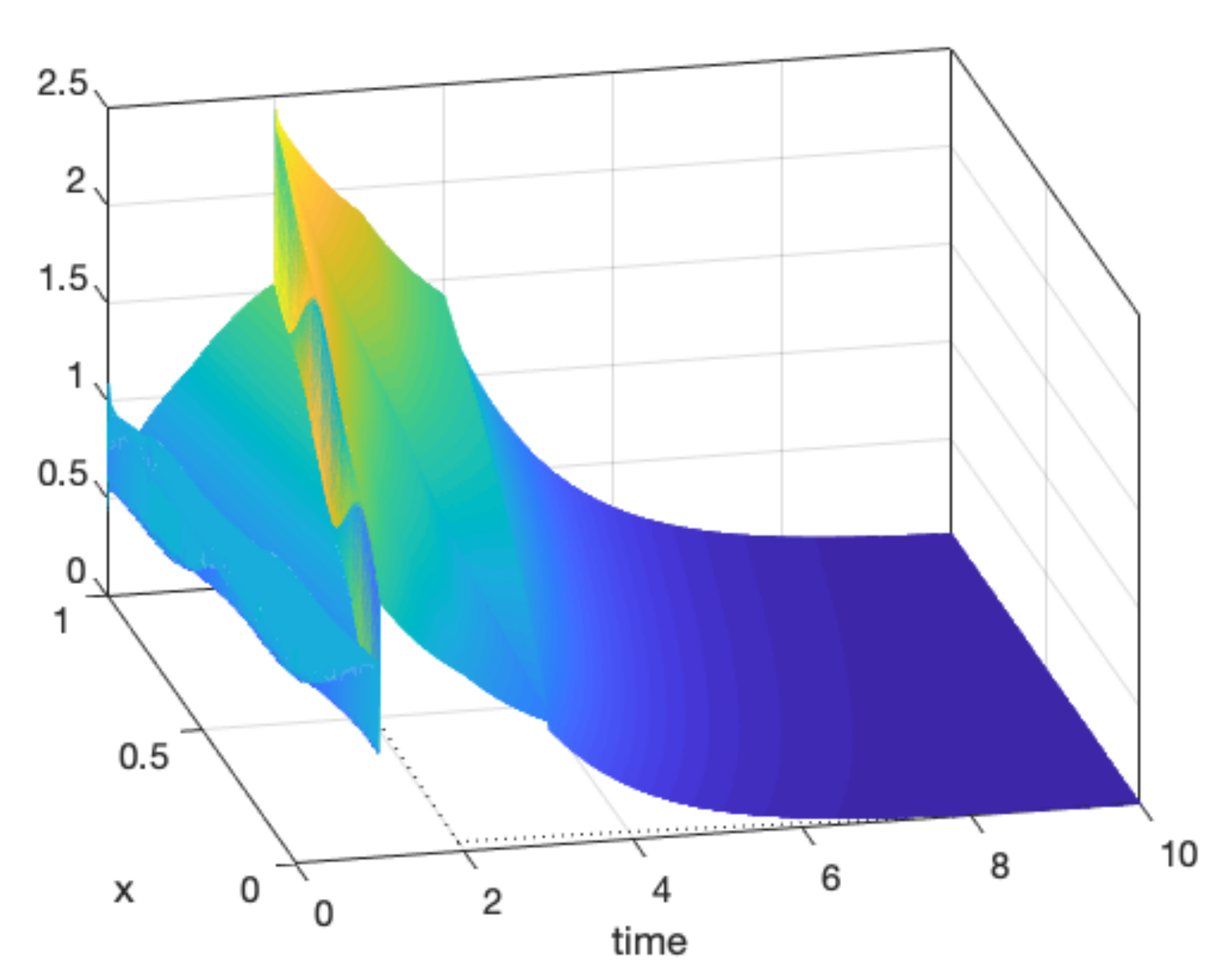}
\caption{Method 2. $\widetilde{K}(q) = \widetilde{K}_{\rm nom} + (q-3)\widetilde{K}_1 + (q-3)^2 \widetilde{K}_2$ and
$K(q) = K_0+{\tt feedback}(\widetilde{K}(q), \Phi(q)$. Simulations for $q=2,3,4$ \label{method2}}
\end{figure}

In conclusion, the study of the hyperbolic system (\ref{wave})
shows that optimization based on the infinite-dimensional program (\ref{program}) is required to synthesize finite-dimensional
controllers for (\ref{wave}), while its finite-dimensional counterpart based on \cite{AN06a} and implemented in {\tt systune} is sufficient
to synthesize infinite-dimensional controllers of the structure covered by  Fig. \ref{swap1}. The major difference with parabolic
systems or first-order hyperbolic systems (see e.g. \cite{ANR:2015}) is that preliminary structured stabilization, based on a suitable heuristic,
cannot be verified using the Nyquist test. A very first stabilizing controller has to be found by way of some
other means, but once this is achieved, the Nyquist test can be brought back to serve to control stability of the loop during optimization.

%\phantomsection
%\bibliographystyle{elsarticle-num}
%\bibliography{biblio_old}

%\bibliography{DATABASE} 
%\bibliographystyle{IEEEtran}

\end{document}